\documentclass[12pt]{article}
\usepackage{latexsym,amsfonts,amsmath,amsthm,amssymb,amscd, bbm}
\usepackage{graphicx}
\usepackage{verbatim}
\usepackage{tikz,tikz-cd}
\usepackage{enumitem}
\scrollmode

\newtheorem{theorem}{Theorem}
\newtheorem{lemma}[theorem]{Lemma}
\newtheorem{corollary}[theorem]{Corollary}

\newtheorem{exmp}[theorem]{Example}

\newtheorem{remark}[theorem]{Remark}
\newtheorem{example}[theorem]{Example}

\newcommand{\bsgamma}{{\boldsymbol{\gamma}}}

\newcommand{\bst}{{\boldsymbol{t}}}

\newcommand{\R}{\mathbb{R}}
\newcommand{\N}{\mathbb{N}}
\newcommand{\Z}{\mathbb{Z}}

\newcommand{\mask}[1]{}

\newcommand{\by}{{\mathbf y}}
\newcommand{\bg}{{\boldsymbol{\gamma}}}

\newcommand{\X}{{\mathfrak X}}

\newcommand{\nn}{\ast}

\newcommand{\kl}{K_{\rm{low}}}
\newcommand{\kld}{K^d_{\rm{low}}}

\newcommand{\ku}{K_{\rm{up}}}
\newcommand{\kud}{K^d_{\rm{up}}}
\newcommand{\kua}{K^{\alpha}_{\rm{up}}}

\DeclareMathOperator{\op}{op}

\DeclareMathOperator{\Prob}{{\bf P}}
\DeclareMathOperator{\Expec}{{\bf E}}
\DeclareMathOperator{\cardet}{card^{det}}
\DeclareMathOperator{\carran}{card^{ran}}

\DeclareMathOperator{\bx}{{\bf x}}
\DeclareMathOperator{\bl}{{ \bf l}}

\DeclareMathOperator{\supp}{{supp}}

\DeclareMathOperator{\decay}{decay}
\DeclareMathOperator{\ran}{ran}
\DeclareMathOperator{\lin}{lin}

\DeclareMathOperator{\cost}{cost}

\DeclareMathOperator{\deter}{det}

\title{Explicit error bounds for randomized {S}molyak algorithms and an application to infinite-dimensional integration}

\author{Michael Gnewuch\thanks{Institut f\"ur Mathematik, Universit\"at Osnabr\"uck,
Germany ({\tt michael.gnewuch@uni-osnabrueck.de})}
\and Marcin Wnuk\thanks{Mathematisches Seminar, Christian-Albrechts-Universit\"at zu Kiel,
Germany ({\tt wnuk@math.uni-kiel.de}).}}

\begin{document}


\date{\today}

\maketitle

\begin{abstract}
Smolyak's method, also known as hyperbolic cross approximation or sparse grid method, is a powerful 
tool to tackle multivariate tensor product problems solely
with the help of efficient algorithms for the corresponding univariate problem.

In this paper we study the randomized setting, i.e., we randomize Smolyak's method.
We provide upper and lower error bounds for randomized Smolyak algorithms with explicitly given dependence on the number of variables and the number of information evaluations used.
The error criteria we consider are the worst-case root mean square error (the typical error criterion for randomized algorithms, often referred to as ``randomized error'') and the root mean square worst-case error (often referred to as ``worst-case error'').

Randomized Smolyak algorithms can be used as building blocks for efficient methods such as multilevel algorithms, multivariate decomposition methods or dimension-wise quadrature methods to tackle successfully high-dimensional or even infinite-dimensional problems. As an example, we provide a very general and sharp result on the convergence rate of $N$-th minimal errors of infinite-dimensional integration on weighted reproducing kernel Hilbert spaces. Moreover, we are able to characterize the spaces for which randomized algorithms  for infinte-dimensional integration are superior to deterministic ones.  We illustrate our findings for the special instance of weighted Korobov spaces. We indicate how these results can be extended, e.g., to spaces of functions whose smooth dependence on successive variables increases (``spaces of increasing smoothness'') and to the problem of $L^2$-approximation (function recovery).
\end{abstract}

\section{Introduction}\label{SECT1}

Smolyak's method or algorithm, also known as  sparse grid method, hyperbolic cross approximation, Boolean method, combination technique or discrete blending method,  was outlined by Smolyak in \cite{Smo63}. It is a general method to treat multivariate tensor product problems. Its major advantage is the following: to tackle a multivariate tensor product problem at hand one only has to understand the corresponding univariate problem. More precisely, Smolyak's algorithm uses algorithms for the corresponding univariate problem as building blocks, and it is fully determined by the choice of those algorithms. If those algorithms for the univariate problem are optimal, then typically Smolyak's algorithm  for the multivariate problem is almost optimal, i.e., its convergence rate is optimal up to logarithmic factors.

Today Smolyak's method is widely used in scientific computing and there exists a huge number of scientific articles dealing with applications and modifications of it.
A  partial list of papers (which is, of course, very far from being complete) on \emph{deterministic Smolyak algorithms} may contain, e.g., the articles
\cite{WW95, WW99} for general approximation problems,  \cite{Gen74, Del90, Tem92, BD93, FH96, NR96, GG98, GG03, Pet03, GLSS07, HHPS18} for numerical integration, \cite{Gor71, Del82, Tem87, Tem93, SU07, Ull08, DG16} for function recovery, and \cite{Per86, Zen91, NR96a, Yse05, Yse06, Gar07, GH09} for other applications.
Additional references and further information may be found in the survey articles \cite{BG04, Gri06}, the book chapters \cite[Chapter~15]{NW10}, \cite[Chapter~4]{Tem18}, and the books \cite{DS89, DTU18}.



On \emph{randomized Smolyak algorithms} much less is known. Actually, we are only aware of two articles that deal with randomized versions of Smolyak's method, namely \cite{DLP07} and  \cite{HM11}.  In \cite{DLP07} Dick et al. investigate a specific instance of the randomized Smolyak method and use it as a tool to show that higher order nets may be used to construct integration algorithms achieving almost optimal order of convergence (up to logarithmic factors) of the worst case error in certain Sobolev spaces. In \cite{HM11} Heinrich and Milla employ the randomized Smolyak method as a building block of an algorithm to compute antiderivatives of functions from $L^p([0,1]^d),$ allowing for fast computation of antiderivative values for any point in $[0,1]^d.$ Note that in both cases the randomized Smolyak method is applied as an ad hoc device and none of the papers gives a systematic treatment of its properties.

With this paper we want to start a systematic treatment of randomized Smolyak algorithms.
Similar to the paper \cite{WW95}, where deterministic Smolyak methods were studied, we discuss the randomized Smolyak method for general linear approximation problems on
tensor products of Hilbert spaces. Examples of such approximation problems are numerical integration or  $L^2$-approximation, i.e., function recovery.

The error criteria for randomized algorithms or, more generally, randomized operators that we consider are extensions of the worst-case error for deterministic algorithms. The first error criterion is the worst-case root mean square error, often referred to as ``randomized error''.  This error criterion is typically used to assess the quality of randomized algorithms.
The second error criterion is the root mean square worst-case error, often referred to as ``worst-case error''. This quantity is commonly used to prove the existence of a good \emph{deterministic} algorithm with the help of the ``pidgeon hole principle'':
It arises as an average of the usual deterministic worst case error over a set of deterministic algorithms $\mathcal{A}$ endowed with a probability measure $\mu$. If the average is small, there exists at least one algorithm in  $\mathcal{A}$ with small worst-case error, see, e.g.,  \cite{DLP07} or \cite{SW98}. Notice that
the pair $(\mathcal{A}, \mu)$ can be canonically identified with a randomized algorithm.

We derive upper error bounds for both error criteria for randomized Smolyak algorithms with explicitly given dependence on the number of variables and the number of information evaluations used.
The former number is the underlying dimension of the problem, the latter number is typically proportional to the cost of the algorithm. The upper error bounds show that the randomized Smolyak method can be  efficiently used at least in moderately high dimension.
We complement this result by providing lower error bounds for randomized Smolyak
algorithms that nearly match our upper bounds.

As in the deterministic case, our upper and lower error bounds contain logarithmic factors whose powers depend linearly on the underlying dimension $d$, indicating that the direct use of the randomized Smolyak method in very high dimension may be prohibitive. Nevertheless, our upper error bounds shows that randomized Smolyak algorithms make perfect building blocks for more sophisticated algorithms such as multilevel algorithms (see, e.g., \cite{Hei98, HS99, Gil08a, GHM09, GiW09, HMNR10, NHMR11, Gne12, BG12, DG12, KSS15}), multivariate decomposition methods (see, e.g., \cite{KSWW10a, PW11, Was12, Gne13, DG12, DG13}) or dimension-wise quadrature methods (see \cite{GH10}).  We demonstrate this in the case of the infinite-dimensional integration problem on weighted tensor products of reproducing kernel Hilbert spaces with general kernels.
We provide the exact polynomial convergence rate of $N$-th minimal errors---the corresponding upper error bound is established by multivariate decomposition methods based on randomized Smolyak algorithms.

The paper is organized as follows: In Section \ref{SECT2} we provide the general multivariate problem formulation and illustrate it with two examples. In Section \ref{SECT3} we introduce the randomized multivariate Smolyak method building on randomized univariate algorithms. Our assumptions about the univariate randomized algorithms resemble the ones made
in \cite{WW95} in the deterministic case.
In Remark~\ref{algL2Rep} we observe that we may identify
our randomized linear approximation problem of interest with a corresponding deterministic $L^2$-approximation problem.

In Section \ref{SECT4} we follow the course of \cite{WW95} and establish first error bounds in terms of the underlying dimension of the problem and the level of the considered Smolyak algorithm,
see Theorem~\ref{BasicLemma} and Corollary~\ref{levelBoundGeneralConditions}.
For the randomized error criterion Remark~\ref{algL2Rep} helps us
to boil the error analysis of the randomized Smolyak method down to the error analysis of the deterministic Smolyak method provided in \cite{WW95}.
For the worst case error criterion Remark~\ref{algL2Rep} is of no help and therefore we
state the details of the analysis.

Up to this point we consider general randomized operators to approximate the solution we are seeking for. In Section \ref{SECT5} we focus on randomized algorithms and the information evalutions they use.
In Theorem~\ref{Theo_UB_PW} we present upper error bounds for the randomized Smolyak method where the dependence on the underlying dimension of the problem and on the number of information evalutions is revealed.
In Corollary \ref{Cor_Lower_Bound} we provide lower error bounds for the randomized Smolyak method.

In Section \ref{INF_DIM_INT} we apply our upper error bounds for randomized Smolyak algorithms to the infinite integration problem. After introducing the setting, we provide the exact polynomial convergence rate of $N$-th minimal errors in
Theorem~\ref{Theo_UB_PW}. In Corollary~\ref{Power}
we compare the power of randomized algorithms and deterministic algorithms for infinite-dimensional integration, and in Corollary~\ref{Korobov}
we illustrate the result of Theorem~\ref{Theo_UB_PW} for weighted Korobov spaces.
In Remark~\ref{Baustelle}
and Remark~\ref{Generalizations}
we discuss previous contributions to the considered infinite integration problem
and generalizations to other settings such as to function spaces with increasing smoothness or to the $L^2$-approximation problem.

In the appendix we provide for the convenience of the reader a self-contained proof of a folklore result on the convergence rates of randomized algorithms on Korobov spaces.

\section{Formulation of the Problem}\label{SECT2}

Let $d\in \N$. For $n=1,\ldots,d,$ let $F^{(n)}$ be a separable Hilbert space of real valued functions, $G^{(n)}$ be a separable Hilbert space, and $S^{(n)}:F^{(n)} \to  G^{(n)}$ be a continuous linear operator.
We consider now  the tensor product spaces $F_d$ and $G_d$ given by
$$F_d := F^{(1)} \otimes \cdots \otimes F^{(d)},$$
$$G_d := G^{(1)} \otimes \cdots \otimes G^{(d)},$$
and the tensor product operator $S_d$ (also called \emph{solution operator}) given by
$$ S_d := S^{(1)} \otimes \cdots \otimes S^{(d)}.$$
We frequently  use results concerning tensor products of Hilbert spaces and tensor product operators without giving explicit reference, for details on this subject see, e.g.,  \cite{Wei80}.
We denote the norms in $F^{(n)}$ and $F_d$ by $\|\cdot \|_{F^{(n)}}$ and $\|\cdot \|_{F_d}$ respectively, and the
norms in $G^{(n)}$ and $G_d$ simply by $\|\cdot\|$. Furthermore, $L(F_d, G_d)$ denotes the space of all bounded linear operators between $F_d$ and $G_d.$

 $S_d(f)$ may be approximated by randomized linear algorithms or, more generally, by randomized linear operators. We define a \emph{randomized  linear operator} $A$ to be a mapping
$$A: \Omega \rightarrow L(F_d, G_d)$$ such that $Af: \Omega \rightarrow G_d$ is a random variable for each $f \in F_d$; here $(\Omega, \Sigma, \Prob)$  is some probability space and $G_d$ is endowed with its Borel $\sigma-$field. We put
$$\mathcal{O}^{\ran} := \mathcal{O}^{\ran, \lin}(\Omega, F_d, G_d) := \{A:\Omega \rightarrow L(F_d, G_d) \, |  \, A \text{ is a randomized linear operator}\}. $$
Obviously one may interpret deterministic bounded linear operators as randomized linear operators with trivial dependance on $\Omega$.
Accordingly, we put
$$\mathcal{O}^{\deter} := \mathcal{O}^{\deter, \lin}(F_d, G_d) := L(F_d, G_d) \subset \mathcal{O}^{\ran, \lin}(\Omega, F_d, G_d),$$
where the inclusion is based on the identification of $A\in L(F_d, G_d)$ with the constant mapping $\Omega \ni \omega \mapsto A$.

A \emph{(randomized) linear approximation problem } is given by a quadruple $\{S_d, F_d, G_d, \mathcal{O}(\Omega) \},$ where $\mathcal{O}(\Omega) \subseteq \mathcal{O}^{\ran,\lin}(\Omega, F_d, G_d)$ denotes the class of admissible randomized linear operators.
 We are mainly interested in results for randomized linear algorithms, which constitute a subclass of $\mathcal{O}^{\ran}$ and will be introduced in Section \ref{SECT5}.

Consider a randomized linear operator $A$ meant to approximate $S_d$. The \emph{randomized error} of the operator is given by
 \begin{equation} \label{ran_err}
 e^{\rm{r}}(A) := e^{\rm{r}}(S_d,A):=  \sup\limits_{\substack{\lVert f \rVert_{F_d} \leq 1}}\Expec \bigg[ \lVert (S_d-A)f  \rVert^2\bigg]^{\frac{1}{2}},
 \end{equation}
and the\emph{ (root mean square) worst case error} is
\begin{equation} \label{wor_err}
e^{\rm{w}}(A) := e^{\rm{w}}(S_d,A):=  \Expec \bigg[ \sup\limits_{\substack{\lVert f \rVert_{F_d} \leq 1}}\lVert (S_d-A)f  \rVert^2\bigg]^{\frac{1}{2}}.
\end{equation}
Clearly we have $0\le e^{\rm{r}}(S_d,A) \le e^{\rm{w}}(S_d,A)$.

Notice that for a \emph{deterministic} linear Operator $A$ both errors coincide with the
\emph{deterministic worst case error}
$$e^{\rm d}(A) := e^{\rm d}(S_d, A) := \sup\limits_{\substack{\lVert f \rVert_{F_d} \leq 1}}\lVert (S_d-A)f  \rVert, $$
i.e., $e^{\rm d}(S_d, A) = e^{\rm r}(S_d, A) = e^{\rm w}(S_d, A)$.

We finish this section by giving two examples of typical tensor product problems that fit into the framework given above.

\begin{example}\label{ProblemExample}
For $n=1,\ldots, d$ let $D^{(n)}\neq \emptyset$ be an arbitrary domain and let $\rho^{(n)}$ be a positive measure on $D^{(n)}$. Denote by $E$ the Cartesian product $D^{(1)} \times \cdots \times D^{(d)}$ and by $\mu$ the product measure $\otimes_{n=1}^d \rho$ on $E$.
\begin{itemize}
\item[{\rm (}i{\rm)}] By choosing  $F^{(n)} \subset L^2(D^{(n)},\rho^{(n)})$, $G^{(n)} := \R$, and $S^{(n)}$ to be the integration functional $S^{(n)}(f) = \int_{D^{(n)}} f \,{\rm d}\rho^{(n)}$, we obtain $F_d \subset L^2(E, \mu)$, $G_d = \R$, and $S_d$ is the integration functional on $F_d$ given by
$$S_d(f) = \int_{E} f \,{\rm d}\mu, \hspace{3ex} f\in F_d.$$
The \emph{integration problem} is now to compute or approximate for given $f\in F_d$ the integral $S_d(f)$.
\item[{\rm (}ii{\rm)}] By choosing $F^{(n)} \subset G^{(n)} := L^2(D^{(n}),\rho^{(n)})$ and $S^{(n)}$ to be the embedding operator from $F^{(n)}$ into $G^{(n)}$, we obtain $F_d \subset G_d = L^2(E, \mu)$
and $S_d$ is the embedding operator from $F_d$ into $G_d$ given by
$$S_d(f) = f, \hspace{3ex} f\in F_d.$$
The \emph{$L^2$-approximation problem} is now to reconstruct  a given function $f\in F_d$, i.e., to compute or approximate $S_d(f)$; the reconstruction error is measured in the $L^2$-norm.
\end{itemize}
Note that in both problem formulations above the phrase ``a given function $f$'' does not necessarilly mean that the whole function is known. Usually there is only partial information about the function available (like a finite number of values of the function or of its derivatives or a finite number of Fourier coefficients) available. We discuss this point in more detail in
Section \ref{ALG5.1}.
\end{example}

\section{Smolyak Method for Tensor Product Problems}\label{SECT3}

From now on we are interested in randomizing the Smolyak method which is to be defined in this section. Assume that for every $n = 1,2,\ldots, d,$ we have a sequence of randomized linear operators $(U_l^{(n)})_{ l \in \mathbb{N}},$ which approximate the solution operator $S^{(n)}$
such that for every $f \in F^{(n)}$ it holds: $U_l^{(n)}f$ is a random variable on a probability space $(\Omega^{(n)}, \Sigma^{(n)}, \Prob^{(n)}).$ We shall refer to separate $U^{(n)}_l$ as to \emph{building blocks}.

Put $\Omega := \Omega^{(1)} \times \ldots \times \Omega^{(d)}, \Sigma := \bigotimes_{n = 1}^{d} \Sigma^{(n)}, \Prob := \bigotimes_{n = 1}^{d} \Prob^{(n)}$. We denote
\begin{align*}
 & \Delta_0^{(n)} := U_0^{(n)} := 0, \quad  \Delta_l^{(n)} := U_l^{(n)} - U_{l-1}^{(n)}, \quad  l \in \mathbb{N},
\end{align*}
and
\begin{align*}
 & Q(L,d) := \left\{\bl \in \mathbb{N}^d \, | \, |\bl| \leq L\right\}.
\end{align*}
Note that if $L \geq d$, then $|Q(L,d)| = \binom{L}{d}$.
For $f \in F_d $ the \emph{randomized Smolyak method of level L} approximating the tensor product problem $\{S_d,F_d,G_d, \mathcal{A}(\Omega)\}$ is given by

\begin{equation} \label{basic_form}
A(L,d)f:\Omega \rightarrow G_d, \quad \omega \mapsto \left(\sum\limits_{\bl \in Q(L,d) } \bigotimes_{n = 1}^{d}\Delta^{(n)}_{l_n}(\omega_n)\right)f.
\end{equation}
We would like to stress that due to the definition of the probability space $(\Omega, \Sigma, \Prob)$ for given $f_n \in F^{(n)}, n = 1,2, \ldots, d,$ the families $((U^{(n)}_lf_n)_{l \in \mathbb{N}}), n = 1,2, \ldots, d,$ are mutually independent.
 Note that for $L < d$ the Smolyak method is the zero operator. Therefore, we will always assume (without stating it explicitly every time) that $L \geq d.$

It can be verified  that the following representation holds
\begin{equation}\label{alg_bb_rep}
A(L,d) = \sum_{L-d+1 \leq |\bl| \leq L} (-1)^{L-|\bl|} \binom{d-1}{L-|\bl|} \bigotimes_{n = 1}^d U^{(n)}_{l_n},
\end{equation}
cf. \cite[Lemma 1]{WW95}.
When investigating the randomized error we  need that for every $l \in \mathbb{N}$ and $n=1,\ldots,d$
\begin{equation}\label{Ran1}
 U_l^{(n)}f \in L^{2}(\Omega^{(n)}, G^{(n)}) \quad \text{ for all } f \in F^{(n)}.
\end{equation}
In the worst case error analysis we require for every $l \in \mathbb{N}$ and $n=1,\ldots,d$
\begin{equation}\label{Wce1}
 \mu_{l,n}(\omega) := \sup\limits_{\substack{\lVert f \rVert_{F^{(n)}} \leq 1}} \lVert (U_l^{(n)}f)(\omega_n) \rVert < \infty \quad \text{ for all } \omega_n \in \Omega^{(n)}
\end{equation}
and that $\mu_{l,n}: \Omega  \rightarrow [0, \infty)$ is measurable with
\begin{equation}\label{Wce2}
 \lVert \mu_{l,n} \rVert_{L^{2}(\Omega^{(n)}, \mathbb{R})} < \infty.
\end{equation}

Let $\rm{x} \in \{\rm{r}, \rm{w}\}.$ When considering the error $e^{\rm{x}}(S_d,A(L,d)),$ we assume that there exist constants $B,C,E > 0$ and $D \in (0,1)$ such that for every $n = 1,2, \ldots, d,$ and every $l \in \mathbb{N}$

\begin{equation} \label{IntNorm}
\lVert S^{(n)} \rVert_{\rm{op}}  \leq B,
\end{equation}

\begin{equation} \label{ApproxNorm}
 e^{\rm{x}}(S^{(n)}, U_l^{(n)}) \leq CD^l,
\end{equation}
and additionally in the randomized setting
\begin{equation} \label{DiffNorm}
 \sup\limits_{\substack{ \lVert f \rVert_{F^{(n)}} \leq 1}}\Expec \bigg[\lVert \underbrace{ U_l^{(n)}f - U_{l - 1}^{(n)}f}_{= \Delta^{(n)}_lf} \rVert^2\bigg]^{\frac{1}{2}} \leq E D^l,
\end{equation}
and in the worst case setting





\begin{equation} \label{DiffNormWce}
\Expec \bigg[ \sup\limits_{\substack{ \lVert f \rVert_{F^{(n)}} \leq 1}}\lVert \underbrace{ U_l^{(n)}f - U_{l - 1}^{(n)}f}_{= \Delta^{(n)}_lf} \rVert^2\bigg]^{\frac{1}{2}} \leq E D^l.
\end{equation}
Note that (\ref{ApproxNorm}) implies the conditions (\ref{DiffNorm}) and (\ref{DiffNormWce}) with a constant $E := C(1+D^{-1})$ for all $l \geq 2.$
Still (\ref{DiffNorm}) and (\ref{DiffNormWce}) may even hold for some smaller $E.$

\begin{remark}\label{algL2Rep}
For our randomized error analysis it would be 
convenient to identify a randomized linear operator $A: \Omega \rightarrow L(F,G),$ $F,G$ separable Hilbert spaces, with the mapping (\ref{algo_form}) which we again denote by $A:$
\begin{equation}\label{algo_form}
A : F \to L^2(\Omega, G), \quad  f\mapsto \big( \omega \mapsto Af(\omega) \big).
\end{equation}
We now show that this identification makes sense for all the operators we are considering. We start with the building blocks $U^{(n)}_l.$
From (\ref{DiffNorm}) we obtain
\begin{equation}\label{Ran2}
 \sup\limits_{\substack{\lVert f \rVert_{F^{(n)}} \leq 1}}\Expec \left[ \lVert U_l^{(n)}f  \rVert^2
 \right]^{1/2} \leq \frac{ED}{1-D},
\end{equation}
 implying $(U^{(n)}_lf)(\cdot) \in L^2(\Omega^{(n)}, G^{(n)})$ for all $f \in F^{(n)}.$  The building blocks $U_l^{(n)}$ are obviously linear as mappings $F^{(n)} \rightarrow L^2(\Omega^{(n)}, G^{(n)})$ and, due to (\ref{Ran2}), also bounded, i.e. continuous.
 Now, since for arbitrary sample spaces $ \Omega_1, \Omega_2$  and separable Hilbert spaces $H_1, H_2$ it holds $$L^{2}(\Omega_1, H_1) \otimes L^{2}(\Omega_2, H_2) \cong L^{2}(\Omega_1 \times \Omega_2, H_1 \otimes H_2),$$ we have that  $(\bigotimes_{n = 1}^{d} U^{(n)}_{l_n})(f)(\cdot) $ lies in $L^{2}(\Omega,G_d)$ for $f \in F_d.$ Clearly, the tensor product operator $\bigotimes_{n = 1}^{d} U^{(n)}_{l_n}$ is a bounded linear mapping $F\to L^2(\Omega, G_d)$. Since due to \eqref{alg_bb_rep} the Smolyak method $A(L,d)$ may be represented as a finite sum of such tensor product operators, it is also a bounded linear map $F\to L^{2}(\Omega,G_d).$

If we formally consider $S_d$ as an operator $F_d\to L^2(\Omega, G_d)$, $f\mapsto \big( \omega \mapsto S_df \big)$ (i.e., an operator that maps into the constant $L^2$-functions),
then $S_d$ is still linear and continuous with operator norm
\begin{equation*}
\|S_d\|_{\op} =  \sup_{\|f\|_{F_d}\le 1} \| S_df\|_{L^2(\Omega, G)} = \sup_{\|f\|_{F_d}\le 1} \Expec \left[ \| S_df(\omega)\|^2 \right]^{1/2},
\end{equation*}
and the usual randomized error
 can be written as
\begin{equation}\label{error_op_norm}
e^{\rm{r}}(S_d, A)   = \sup_{\|f\|_{F_d}\le 1} \|(S_d-A)f\|_{L^2(\Omega, G_d)} = \|S_d- A\|_{\op}.
\end{equation}
The  worst case error
unfortunately does not allow for a representation as operator norm similar to  \eqref{error_op_norm}.

Note that the above identification turns a randomized approximation problem
$$S:F\rightarrow G, \quad A: \Omega \rightarrow L(F,G) $$
into a deterministic $L^2$-approximation problem
$$S: F \rightarrow L^2(\Omega, G), \quad A:F \rightarrow L^2(\Omega, G).$$
\end{remark}

\section{Error Analysis in Terms of the Level}\label{SECT4}
We now perform the error analysis of the approximation of $S_d$ by the Smolyak method $A(L,d)$ in terms of the level $L,$ which may be done under the rather general assumptions of Sections $2$ and $3.$

\begin{theorem}\label{BasicLemma}
For $L,d \in \mathbb{N}, L\geq d$ let $A(L,d)$ be a randomized Smolyak method as described in Section \ref{SECT3}. Let $\rm{x} \in \{\rm{w}, \rm{r}\}$. Assume  \eqref{IntNorm}, \eqref{ApproxNorm} and, dependently on the setting, for $\rm{x} = \rm{r}$ additionally assume (\ref{Ran1}),(\ref{DiffNorm}) and for $x = \rm{w}$ additionally assume (\ref{Wce1}), (\ref{Wce2}) and \eqref{DiffNormWce}. Then we have

\begin{equation}\label{levelBound}
e^{\rm{x}}(S_d,A(L,d)) \leq  CB^{d-1}D^{L-d+1} \sum\limits_{j = 0}^{d-1} \bigg( \frac{ED}{B} \bigg)^j \binom{L-d+j}{j} \leq CH^{d-1}\binom{L}{d-1}D^L,
\end{equation}

where $H = \max\{\frac{B}{D}, E\}.$
\end{theorem}

\begin{proof}
The second inequality in (\ref{levelBound}) follows easily by using $\sum_{j = 0}^{d-1} \binom{L-d+j}{j} = \binom{L}{d-1}$ and estimating $(\frac{ED}{B})^j \leq \max \{1, (\frac{ED}{B})^{d-1}\},$ so all there remains to be done is proving the first inequality.

Firstly we shall focus on the worst case error bound. Note that for a fixed $\omega \in \Omega$
$$\sup\limits_{\substack{\lVert f \rVert_{F_d} \leq 1}} \lVert (S_d - A(L,d)(\omega))f \rVert^2 = \left(\sup\limits_{\substack{\lVert f \rVert_{F_d} \leq 1}} \lVert (S_d - A(L,d)(\omega))f \rVert\right)^2 = \lVert S_d - A(L,d)(\omega) \rVert^2_{{\rm op}}$$
Now we may proceed similarly as in the proof of Lemma $2$ from \cite{WW95}, by induction on $d,L$ for $d \in \mathbb{N}$ and $L \in \{d,d+1,\ldots\}$ . For $d=1$ and any  $L \in \mathbb{N}_{\geq d}$ we have $S_d = S^{(1)}$ and $A(L,1) = U^{(1)}_L,$ so the statement is just the condition (\ref{ApproxNorm}). Suppose we have already proved the claim for $L,d$ and want to prove it for $L+1,d+1.$ Using $$A(L+1,d+1) = \sum_{\bl \in Q(L,d)} \bigotimes_{n = 1}^{d} \Delta^{(n)}_{l_n} \otimes U^{(d+1)}_{L+1-|\bl|}$$ and Minkowski's inequality we get
\begin{align*}
&\ e^{\rm{w}}(S_{d+1}, A(L+1, d+1)) = \Expec \bigg[\lVert  S_{d+1} - A(L+1, d+1)  \rVert^2_{{\rm op}}   \bigg]^{\frac{1}{2}}
\\
&= \Expec\bigg[\lVert \sum_{\bl \in Q(L,d)} (\bigotimes_{n = 1}^d \Delta_{l_n}^{(n)}) \otimes (S^{(d+1)} - U^{(d+1)}_{L+1-|\bl|}) + (S_d - A(L,d))\otimes S^{(d+1)} \rVert^2_{{\rm op}}  \bigg]^{\frac{1}{2}}
\\
&\leq  \Expec\bigg[\lVert \sum_{\bl \in Q(L,d)} (\bigotimes_{n = 1}^d \Delta_{l_n}^{(n)}) \otimes (S^{(d+1)} - U^{(d+1)}_{L+1-|\bl|}) \rVert^2_{{\rm op}}  \bigg]^{\frac{1}{2}} + \Expec\bigg[\lVert    (S_d - A(L,d))\otimes S^{(d+1)}      \rVert^2_{{\rm op}}  \bigg]^{\frac{1}{2}}.
\end{align*}

We use Minkowski's inequality, properties of tensor product operator norms, the fact that component algorithms $U^{(n)}_l, l \in \mathbb{N}$ are randomized independently for different $n \in \{1,\ldots, d\}$, (\ref{ApproxNorm}) and (\ref{DiffNormWce})  to obtain
\begin{align*}
\begin{split}
&  \Expec  \bigg[\lVert \sum_{\bl \in Q(L,d)} (\bigotimes_{n = 1}^d \Delta_{l_n}^{(n)}) \otimes (S^{(d+1)} - U^{(d+1)}_{L+1-|\bl|}) \rVert^2_{{\rm op}}  \bigg]^{\frac{1}{2}}
\\
&\leq \sum_{\bl \in Q(L,d)} \Expec \bigg[ \lVert \bigotimes_{n = 1}^{d} \Delta^{(n)}_{l_n}\rVert^{2}_{{\rm op}} \lVert S^{(d+1)} - U^{(d+1)}_{L+1-|\bl|} \rVert^{2}_{{\rm op}}   \bigg]^{\frac{1}{2}}
\\
&= \sum_{\bl \in Q(L,d)} \bigg(\prod_{n = 1}^d  \Expec \bigg[ \lVert \Delta_{l_n}^{(n)}\rVert_{{\rm op}}^2 \bigg]^{\frac{1}{2}} \bigg) \Expec\bigg[ \lVert S^{(d+1)} - U_{L+1-|\bl|}^{(d+1)} \rVert_{{\rm op}}^2\bigg]^{\frac{1}{2}}
\\
&\leq \sum_{\bl \in Q(L,d)} CE^dD^{L+1} = \binom{L}{d}CE^d D^{L + 1}.
\end{split}
\end{align*}
Furthermore, using (\ref{IntNorm}),
\begin{align*}
& \Expec\bigg[\lVert    (S_d - A(L,d))\otimes S^{(d+1)}      \rVert^2_{{\rm op}}  \bigg]^{\frac{1}{2}}
\\
& = \Expec \bigg[\lVert    (S_d - A(L,d))\rVert^2_{{\rm op}} \lVert S^{(d+1)}      \rVert^2_{{\rm op}}  \bigg]^{\frac{1}{2}}
\\
& =  \lVert S^{(d+1)}      \rVert_{{\rm op}} \Expec \bigg[\lVert    (S_d - A(L,d))\rVert^2_{{\rm op}} \bigg]^{\frac{1}{2}}
\\
& \leq Be^{{\rm w}}(S_d,A(L,d)).
\end{align*}
Therefore we have
\begin{align*}
&  e^{\rm{w}}(S_{d+1}, A(L+1,d+1))  \leq   \binom{L}{d} E^dCD^{L+1} + Be^{{\rm w}}(S_d,A(L,d))
\end{align*}
and using the induction hypothesis finishes the proof for the worst case error.

Now consider the randomized error.
By similar calculations as in the first part of the proof one could show that the claim holds true for the randomized error for elementary tensors. Then however, one encounters problems trying to lift it to the whole Hilbert space. The difficulty lies in the fact that the randomized error is not an operator norm of some tensor product operator, which would have enabled us to write it as a product of norms of the corresponding univariate operators and which has proved to be useful in bounding the worst case error. To get round it we need a different approach. The idea is to interpret a randomized problem as a deterministic $L^2-$approximation problem.
As already explained in the Remark \ref{algL2Rep} we may identify $(S_d - A(L,d)): \Omega \rightarrow L(F_d,G_d)$ with an operator
$F_d \rightarrow L^{2}(\Omega, G_d)$ again denoted by $(S_d - A(L,d)).$
Then however $e^{\rm{r}}(S_d, A)  = \|S_d- A\|_{\op}$ and we may proceed exactly as in Lemma 2 from \cite{WW95}, which finishes the proof.
\end{proof}

We may generalize the result of the Theorem \ref{BasicLemma} by allowing for more flexibility in convergence rates in (\ref{ApproxNorm}), (\ref{DiffNorm}) and (\ref{DiffNormWce}). It can be used to capture additional logarithmic factors in the error bounds for the building blocks algorithms. This turns out to be particularly useful when investigating the error  bounds for Smolyak methods whose building blocks are, e.g., multivariate  quadratures or approximation algorithms, as it is the case in \cite{DLP07}. Suppose namely there exists a constant $D \in (0,1)$ and non decreasing sequences of positive numbers $(C_l)_l, (E_l)_l, l \in \mathbb{N},$ such that for every $l \in \mathbb{N}$

\begin{equation} \label{ApproxNormSeq}
e^{\rm{x}}(S^{(n)}, U_l^{(n)}) \leq C_lD^l, \quad \rm{x} \in \{\rm{r}, \rm{w}\}.
\end{equation}
Moreover, in case of  the randomized error
\begin{equation} \label{DiffNormSeq}
 \sup\limits_{\substack{ \lVert f \rVert_{F^{(n)}} \leq 1}}\Expec \bigg[\lVert \underbrace{ U_l^{(n)}f - U_{l - 1}^{(n)}f}_{= \Delta^{(n)}_lf} \rVert^2\bigg]^{\frac{1}{2}} \leq E_l D^l ,
\end{equation}
and  in  case of the worst case error
\begin{equation} \label{DiffNormWceSeq}
\Expec \bigg[ \sup\limits_{\substack{ \lVert f \rVert_{F^{(n)}} \leq 1}}\lVert \underbrace{ U_l^{(n)}f - U_{l - 1}^{(n)}f}_{= \Delta^{(n)}_lf} \rVert^2\bigg]^{\frac{1}{2}} \leq E_l D^l.
\end{equation}

It is now easy to prove Corollary \ref{levelBoundGeneralConditions} along the lines of the proof of Theorem \ref{BasicLemma}.

\begin{corollary}\label{levelBoundGeneralConditions}
For $L,d \in \mathbb{N}, L\geq d$ let $A(L,d)$ be a randomized Smolyak method as described in Section \ref{SECT3}. Let $\rm{x} \in \{\rm{w}, \rm{r}\}$. Assume  \eqref{IntNorm}, \eqref{ApproxNormSeq} and, dependently on the setting, for $\rm{x} = \rm{r}$ assume (\ref{Ran1}),(\ref{DiffNormSeq}) and for $x = \rm{w}$ assume (\ref{Wce1}), (\ref{Wce2}) and \eqref{DiffNormWceSeq}. Then we have

\begin{equation}
e^{\rm{x}}(S_d,A(L,d)) \leq  C_LB^{d-1}D^{L-d+1} \sum\limits_{j = 0}^{d-1} \bigg( \frac{E_{L-1}D}{B} \bigg)^j \binom{L-d+j}{j} \leq C_LH_L^{d-1}\binom{L}{d-1}D^L,
\end{equation}

where $H_L = \max\{\frac{B}{D}, E_{L-1}\}.$
\end{corollary}
\begin{remark}
Note that applying Corollary \ref{levelBoundGeneralConditions} to the uni- or multivariate building block algorithms error bounds from \cite{DLP07} we may reproduce the error bounds obtained in this paper for the final (higher dimensional) Smolyak method.
\end{remark}

\section{Error Analysis in Terms of Information}\label{SECT5}

\subsection{Algorithms}\label{ALG5.1}
Consider a linear approximation problem given by $\{S,F,G, \mathcal{O}(\Omega)\}.$ The aim of this section is to specify those linear operators that we want to call algorithms and to explain the typical information-based complexity framework for investigating the error of an algorithm in terms of the cardinality of information, for further reference see, e.g., \cite{TWW88}. To this end we shall specify a class of  linear bounded functionals on $F$  called \emph{admissible information functionals} and denoted by $\Lambda$, which will become one more parameter of the approximation problem.
Given a constant $\tau \in \mathbb{N}_0$ and, if $\tau > 0$, a collection of $\lambda_i \in \Lambda, i = 1,\ldots, \tau,$ the \emph{information operator} $\mathcal{N}: F \rightarrow \mathbb{R}^{\max\{\tau,1\}}$ applied to $f \in F$ is determined via

\[ \mathcal{N}(f) =   \left\{
\begin{array}{ll}
      0 & \text{ if } \tau = 0, \\
      (\lambda_1(f),\ldots, \lambda_{\tau}(f)) & \text{else.} \\
\end{array}
\right. \]

Note that we are considering only non-adaptive information, meaning that the information functionals used do not depend on $f \in F.$

A deterministic linear operator $A \in \mathcal{O}^{\deter, \lin}(F,G)$
is called a \emph{deterministic linear algorithm} if it admits a representation
\begin{align*}
A = \phi \circ \mathcal{N},
\end{align*}
where $\mathcal{N}$ is an information operator and $\phi: \mathbb{R}^{\max\{{\tau},1\}} \rightarrow G$ is an arbitrary mapping. We denote the number of information functionals used by the deterministic algorithm $A$ for any input $f \in F$ by  $\cardet(A,F),$ i.e.,
$$\cardet(A,F) := \tau.$$
 We denote the class of deterministic linear algorithms with admissible information functionals $\Lambda$ by $\mathcal{A}^{\deter,\lin}(F,G, \Lambda).$

Let $(V_l)_{l \in \mathbb{N}}$ be an arbitrary sequence of algorithms and let $(\lambda_{l,i})_{i \in [m_l]}$ be the information functionals used by $V_l.$ We say that the sequence $(V_l)_l$  \emph{ uses nested information} if for every $a < b$
$$\{\lambda_{a,i} \,| \, i \in [m_a]  \} \subseteq  \{\lambda_{b,i} \,| \, i \in [m_b]  \}.$$

 A \emph{randomized linear algorithm} $A \in \mathcal{O}^{\ran, \lin}(\Omega, F,G)$ is a mapping
$$A:\Omega \rightarrow \mathcal{A}^{\deter,\lin}(F,G,\Lambda)$$
such that
$\omega \mapsto \cardet(A(\omega), F) \text{ is a random variable.}$

We denote the class of randomized linear algorithms with admissible information functionals $\Lambda$ by  $\mathcal{A}^{\ran,\lin}(\Omega,F,G,\Lambda) =: \mathcal{A}(\Omega, \Lambda).$

For a randomized linear algorithm $A$ we may finally define
$$\carran(A,F) := \Expec \left[ \cardet(A,F) \right].$$

We say that the information used by a sequence of randomized linear algorithms is \emph{nested} if it is nested for each $\omega \in \Omega.$ Note that the information used by $(A(L,d))_{L \geq d}$ is nested.

Now we would like to make some reasonable assumptions on the cost of building blocks of the Smolyak method.
Consider a randomized Smolyak method as described in Section \ref{SECT3} with building blocks being randomized algorithms. Let
$$m_{l,n} := \carran(U^{(n)}_{l}, F^{(n)}).$$
Notice that $m_{0,n} = 0.$
For $d \in \mathbb{N}, L = d,d+1,\ldots$ put
$$N := N(L,d) := \carran(A(L,d), F_d).$$
Let us assume that for every $n \in \{1,\ldots,d\}$ the sequence $(m_{l,n})_{l \in \mathbb{N}}$ is non-decreasing and that there exist constants $1\leq \kl  \leq \ku, 1< K$ such that for every $ n = 1,\ldots,d, l \in \mathbb{N}$ it holds
\begin{equation} \label{m_est_prot}
\kl  K^{l-1}(K - 1) \leq m_{l,n} - m_{l-1,n} \leq \ku K^{l-1}(K - 1).
\end{equation}
Note that this implies
\begin{equation}\label{m_est}
\kl (K^l - 1) \leq m_{l,n} \leq \ku(K^l - 1), \quad l \in \mathbb{N}.
\end{equation}

\begin{example}
Consider the integration problem from Example \ref{ProblemExample}. Let $s \in \mathbb{N}.$ For $n = 1, \ldots, d$ let $ D^{(n)} = [0,1]^s,$  $\rho^{(n)}$ be Lebesgue measure on $[0,1]^s$ and $F^{(n)}$ be some reproducing kernel Hilbert space of functions defined on $[0,1]^s$ (e.g., a Sobolev space with sufficiently high smoothness parameter).

Choose a prime number $b \geq s$ and for $n = 1, \ldots, d,$ and $l \in \mathbb{N}$ let $\mathcal{P}_l^{(n)}$ be a scrambled $(0,l,s)-$net in base $b$ as introduced in \cite{Owe95}. Now
$$U^{(n)}_l : F^{(n)} \rightarrow \mathbb{R}, \quad f \mapsto \frac{1}{b^l} \sum_{x \in \mathcal{P}_l^{(n)}} f(x)$$
is a randomized algorithm. Moreover, if we randomize $(U^{(n)}_l)_{n,l}$ in such a way that the families $(U^{(n)}_l)_l$ are independent then we may use them as building blocks of the Smolyak method and all the results of this paper apply, cf. also \cite{DLP07}.
\end{example}

\subsection{Upper Error Bounds}
Throughout the whole section we require that the assumptions of Theorem \ref{BasicLemma} hold.
Let us define $\alpha := \frac{\log(\frac{1}{D})}{\log(K)},$ where $K$ is as in (\ref{m_est_prot}) and $D$ is as in (\ref{ApproxNorm}).
We define the \emph{polynomial convergence rate of the algorithms} $U^{(n)}_l, l \in \mathbb{N}$ by
\begin{equation}\label{poly_rate_conv}
\mu^{(n)}_{\rm{x}} := \sup\{\delta \geq 0 \, | \, \sup_{l \in \mathbb{N}}  e^{\rm{x}}(S^{(n)}, U^{(n)}_l) m^{\delta}_{l,n} < \infty\}
\end{equation}
where $\rm{x} \in \{\rm{r}, \rm{w}\}.$ It is straightforward to verify that $\alpha \leq \mu^{(n)}_{\rm{x}}$ for every $n.$
 Indeed,we have
 \begin{equation}\label{mu_u_x}
e^{\rm{x}}(S^{(n)}, U^{(n)}_l) \leq \frac{C\kua}{m_{l,n}^{\alpha}},
\end{equation}
because of
\begin{equation}\label{OneDimCardBound}
 \frac{C \kua }{m_{l,n}^{\alpha}} \geq \frac{C \kua }{\kua (K^l - 1)^{\alpha}} \geq \frac{C}{K^{l \alpha}} = CD^l.
\end{equation}
Hence for each $n \in \{1, \ldots, d\}$ the quantity $\alpha$ is a lower bound on the polynomial order of convergence $\mu^{(n)}_{\rm{x}}$ of the algorithms $U^{(n)}_l, l \in \mathbb{N},$ and can be chosen arbitrarily close to $\mu^{(n)}_{\rm{x}}$ if the constants $C$ and $D$ in (\ref{ApproxNorm}) are chosen appropriately.

The aim of this section is to develop upper bounds on the error of $d-$variate Smolyak method in terms of $N,d$ and $\alpha.$ More concretely we prove the following theorem.
\begin{theorem}\label{CardInfUpBound}
Let $\rm{x} \in \{\rm{r}, \rm{w}\}.$ Let $\kl,\ku,K, \alpha$ be as above, and let the assumptions of Theorem \ref{BasicLemma} hold.
 Then there exist constants $C_0, C_1$  such that for all $d \in \mathbb{N}$ and all $L \geq d$ it holds
\begin{equation}\label{pw_req_est_dim1}
e^{\rm{x}}(A(L,1)) \leq C_0C_1 N^{-\alpha}
\end{equation}
and
\begin{equation}\label{pw_req_est}
e^{\rm{x}}(A(L,d)) \leq  C_0 C_1^d  \left(1 + \frac{ \log(N)}{d-1}\right)^{(d-1)(\alpha + 1)}N^{-\alpha}, \,\, d\geq 2,
\end{equation}
where $N = N(L,d)$ is the cardinality of information used by the algorithm $A(L,d).$
\end{theorem}

To prove Theorem \ref{CardInfUpBound} we need a lemma bounding $N(L,d)$ in terms of $\kl,\ku,K, d$ and $L.$

\begin{lemma}\label{NBound}
Let $\kl,\ku,K$ be as above. Put
$$N_l^{\rm{nest}}:= N_l^{\rm{nest}}(L,d) = \kld \bigg(\frac{K - 1}{K}\bigg)^{d}  K^L \binom{L-1}{d-1}, $$
$$N_u :=N_u(L,d) =  \kud \frac{K}{K - 1} K^L \binom{L-1}{d-1},$$
$$N_u^{\rm{nest}} := N_u^{\rm{nest}}(L,d) := \kud \bigg(\frac{K - 1}{K}\bigg)^{d}  K^L \binom{L-1}{d-1}.$$
For every $d \in \mathbb{N}$ and $L \geq d$ it holds
$$N_l^{\rm{nest}}(L,d) \leq N(L,d) \leq N_u(L,d).$$
Moreover, if the building blocks of the Smolyak method use nested information then
$$ N_l^{\rm{nest}}(L,d) \leq N(L,d) \leq N_u^{\rm{nest}}(L,d).$$
\end{lemma}
\begin{proof}
We have
\begin{align*}
 N(L,d)& = \Expec \left[\cardet\left( \sum_{L-d+1 \leq |\bl|\leq L} (-1)^{L - |\bl|} \binom{d-1}{L - |\bl|} \bigotimes_{n = 1}^d U^{(n)}_{l_n}(\omega)\right)   \right] \\
& \leq  \sum_{L-d+1 \leq |\bl|\leq L} \Expec \left[ \cardet\left(\bigotimes_{n = 1}^d U^{(n)}_{l_n}(\omega) \right)  \right]\\
& =  \sum_{L-d+1 \leq |\bl|\leq L} \prod_{n = 1}^{d} \Expec \left[ \cardet\left(U^{(n)}_{l_n}(\omega)\right)  \right] \\
& = \sum_{L-d+1 \leq |\bl|\leq L} \prod_{n = 1}^{d} m_{l_n,n} \leq \kud \sum_{L-d+1 \leq |\bl|\leq L} K^{|\bl|}.
\end{align*}
Now following the steps of \cite[Lemma 7]{WW95} we obtain
\begin{align*}
N(L,d) &  \leq \kud \sum_{|\bl| = L - d +1}^{L} K^{|\bl|} \leq \kud \sum_{\nu = L - d +1}^{L} K^{\nu} \binom{\nu - 1}{d-1}\\
& \leq \kud \binom{L-1}{d-1} (K^{L+1} - K^{ L-d+1 })(K-1)^{-1}\\
& \leq \kud \frac{K}{K - 1} K^{L} \binom{L-1}{d-1} = N_u.
\end{align*}
Now we provide a lower bound on $N(L,d).$
Note that given the cardinality of information used by the building blocks, the cardinality of information used by the Smolyak method is minimal when the information used by the building blocks is nested for every coordinate. In this case the information used by the Smolyak method is exactly the information used by $\sum_{|\bl| = L} \bigotimes_{n = 1}^{d} U^{(n)}_{l_n}.$ Let us fix $\bl \in \mathbb{N}^d, |\bl| = L.$ The expected value of the cardinality of information used by $\bigotimes_{n = 1}^{d}U^{(n)}_{l_n}$ and at the same time not used by any other $\bigotimes_{n = 1}^{d}U^{(n)}_{ v_n}$ with $|{\bf v}| = L$ is
\begin{equation}
\prod_{n = 1}^{d} (m_{l_n,n} - m_{l_n - 1, n}) \geq \kld K^{L-d}(K-1)^d.
\end{equation}
We obtain
\begin{align*}
N(L,d) & \geq  \sum_{|\bl| = L}  \kld K^{L-d}(K-1)^d \\
& = \kld \left( \frac{K - 1}{K}  \right)^d K^L \binom{L-1}{d-1} =N_l^{\rm{nest}}.
\end{align*}
The upper bound in the case when the building blocks use nested information follows in exactly the same manner on noting that
\begin{equation}
\kud K^{L-d}(K-1)^d \geq  \prod_{n = 1}^{d} (m_{l_n,n} - m_{l_n - 1, n}).
\end{equation}
\end{proof}

\begin{proof}{(Theorem \ref{CardInfUpBound})}

Note that $N(L,1) = m_{L,1}$ so we have already showed the statement for $d = 1$ in (\ref{OneDimCardBound}). It remains to consider the case $d > 1.$
Consider the function
$$f: [1, \infty) \rightarrow \mathbb{R}, \quad  x \mapsto \left(1 + \frac{\log(x)}{d-1}\right)^{(d-1)(\alpha + 1)}x^{-\alpha}.$$
We will show that there exist constants $\widetilde{C}_{u,0}, \widetilde{C}_{u,1}, \widetilde{C}_{l,0}, \widetilde{C}_{l,1}$ such that for $N_u,N_l$ from Lemma \ref{NBound} it holds
\begin{equation}\label{NuBound}
e^{\rm{x}}(A(L,d))  \leq \widetilde{C}_{u,0}\widetilde{C}_{u,1}^d f(N_u)
\end{equation}
and
\begin{equation}\label{NlBound}
e^{\rm{x}}(A(L,d))  \leq \widetilde{C}_{l,0}\widetilde{C}_{l,1}^d f(N_l).
\end{equation}
Now unimodality of $f$ combined with the fact that the extremum is a maximum yields $f(N(L,d)) \geq \min \{f(N_u), f(N_l) \}$ finishing the proof.

First we prove (\ref{NuBound}).
Calling upon Theorem \ref{BasicLemma} and using $L \leq \frac{\log(N_u)}{ \log(K)}$ we get

\begin{align*}
e^{\rm{x}}(S_d,A(L,d)) & \leq CH^{d-1} \binom{L}{d-1}D^L
\\
& =CH^{d-1} \binom{L}{d-1}K^{-L\alpha}
\\
& = CH^{d-1}\binom{L}{d-1}\bigg(\kud \frac{K}{K - 1} \bigg)^{\alpha} \binom{L-1}{d-1}^{\alpha}N_u^{-\alpha}
\\
& \leq \frac{C}{H}(H\kua)^d \left(\frac{K}{K - 1}\right)^{\alpha}\binom{L}{d-1}^{\alpha+1}N_u^{- \alpha}
\\
& \leq \frac{C}{H}(H\kua)^d \left(\frac{K}{K - 1}\right)^{\alpha} \frac{\log(N_u)^{(d-1)(\alpha+1)}}{( \log(K))^{(d-1)(\alpha+1)} ((d-1)!)^{\alpha+1}} N_u^{- \alpha}
\\
& = \frac{C}{((d-1)!)^{\alpha+1}H}\left(\frac{K \log(K)}{K - 1}\right)^{\alpha}\log(K) \left( \frac{H \kua}{(    \log(K))^{\alpha+1}} \right)^d \frac{\log(N_u)^{(d-1)(\alpha+1)}}{N_u^{\alpha}}
\\
& = ((d-1)!)^{-\alpha - 1} C_{u,0} C_{u,1}^d \frac{\log(N_u)^{(d-1)(\alpha+1)}}{N_u^{\alpha}},
\end{align*}
with constants $C_{u,0}, C_{u,1}$ not depending neither on $d$ nor on $N(L,d).$ By Stirling's formula we conclude
\begin{align*}
e^{\rm{x}}(S_d,A(L,d)) & \leq \left(\frac{e^d}{(2\pi)^{\frac{1}{2}}(d-1)^{\frac{1}{2}}  (d-1)^{(d-1)}  }\right)^{\alpha + 1}  C_{u,0} C_{u,1}^d \frac{\log(N_u)^{(d-1)(\alpha+1)}}{N_u^{\alpha}} \\
& \leq \widetilde{C}_{u,0} \widetilde{C}_{u,1}^d   \left(\frac{\log(N_u)}{d-1}\right)^{(d-1)(\alpha + 1)}N_u^{-\alpha}.
\end{align*}

Now we prove (\ref{NlBound}).
To this end it suffices to prove that there exist constants $\hat{C}_{0}, \hat{C}_1$ independent of $d$ and $N$ such that
\begin{equation} \label{ulInequality}
 \left(\frac{\log(N_u)}{d-1}\right)^{(d-1)(\alpha + 1)}N_u^{-\alpha} \leq \hat{C}_0 \hat{C}_1^d  \left( 1 + \frac{\log(N_l)}{d-1}\right)^{(d-1)(\alpha + 1)}N_l^{-\alpha}, \end{equation}
i.e.,
$$\left(\frac{N_l}{N_u}\right)^{\alpha} \left(\frac{\log(N_u)}{(d-1) + \log(N_l)}   \right)^{(d-1)(\alpha + 1)} \leq \hat{C}_0 \hat{C}_1^d .$$
Note that
$$\frac{N_u}{N_l} = \left(\frac{K}{K-1} \right)^{d+1} \left(\frac{\ku}{\kl}\right)^d$$
so, putting $\hat{K} = \frac{K}{K-1}\frac{\ku}{\kl}$ we have
\begin{align*}
\frac{\log(N_u)}{(d-1) + \log(N_l)} \leq  & \frac{\log \left( \frac{K}{K-1} \right) + d\log(\hat{K}) + \log(N_l)}{(d-1) + \log(N_l)}
\\
& \leq \log\left( \frac{K}{K-1}\right) + \frac{d}{d-1}\log(\hat{K}) + 1 \leq \log\left( \frac{K}{K-1}\right) + 2\log(\hat{K}) + 1 .
\end{align*}
Since obviously $\left(\frac{N_l}{N_u}\right)^{\alpha} \leq 1$ this shows (\ref{ulInequality}) and finishes the proof of the theorem.
\end{proof}

\subsection{Lower Error Bounds}
In this subsection we make the following additional assumptions.

The first assumption states that there exist a sequence of instances of the problem $\{S_d,F_d,G_d, \mathcal{A}(\Omega, \Lambda)\}$ that is genuinely univariate, i.e.,
there exists a sequence $(f_l)_{l \in \mathbb{N}} \in F_d, \, f_l= g_{1,l} \otimes g_{2,l} \otimes \cdots \otimes g_{d,l}$ such that $\lVert f_l  \rVert_{F_d} = 1$ for which
\begin{equation}\label{v1}
\lVert(S^{(2)} \otimes \cdots \otimes S^{(d)})(g_{2,l} \otimes \cdots \otimes g_{d,l})\rVert =: \theta_d > 0,
\end{equation}
and the $U_l^{(n)}, l\geq 1$ are exact on $g_{n,l}$ for $n>1$ .

Secondly, we assume that there exist constants $\widetilde{C}>0, \widetilde{D}\in(0,1)$ such that  for every $l \in \mathbb{N}$
\begin{equation}\label{v2}
\Expec\bigg[ \lVert S^{(1)}g_{1,l} - U_l^{(1)}g_{1,l} \rVert^2\bigg]^{\frac{1}{2}} \geq \widetilde{C} \widetilde{D}^l.
\end{equation}
Let us put
$$\beta := \frac{\log(\frac{1}{\widetilde{D}})}{\log(K)},$$
with $\widetilde{D}$ as in (\ref{v2}) and $K$ as in (\ref{m_est_prot}).
Using (\ref{v2}) and (\ref{m_est})
one easily sees that
\begin{equation}\label{UnivLowBoundR}
\Expec \left[\lVert (S^{(1)} - U^{(1)}_l)g_{1,l} \rVert^{2} \right]^{\frac{1}{2}} \geq \widetilde{C} \left( \kl(1-K^{-1})   \right)^{\beta} m_{l,1}^{-\beta},
\end{equation}
meaning that we have $\beta \geq \mu^{(1)}_{{\rm{x}}},$ where $ \mu^{(1)}_{{\rm{x}}}$ is as in (\ref{poly_rate_conv}). Moreover, by choosing $(g_{1,l})_{l \in \mathbb{N}}$ appropriately, $\beta$ can be made arbitrarily close to $\mu^{(1)}_{{\rm{x}}}.$
\begin{example}
The assumptions made in this subsection are quite naturally met for many important problems. Consider for instance an integration problem as described in Example \ref{ProblemExample}, where $F^{(n)}, n = 2, \ldots, d$ may be any spaces containing constant functions. Then, for an appropriate  $(g_{1,l})_{l \in \mathbb{N}}$ (chosen so that the integration error does not converge too fast to $0$) we have that
$$f_l := g_{1,l} \otimes \mathbbm{1}_{D} \otimes \cdots \mathbbm{1}_{D}$$
satisfies our assumptions for any randomized quadrature with weights adding up to $1.$
 \end{example}

\begin{lemma}\label{LowBoundLem}
Let ${\rm{x}} \in \{\rm{w}, \rm{r}\},$ and let (\ref{v1}) and (\ref{v2}) hold.
Then there exists a constant $\hat{c}_d$ such that
$$e^{{\rm{x}}}(A(L,d)) \geq \hat{c}_d m_{L,1}^{-\beta}$$
for all $L \geq d.$

If additionally (\ref{v1}) and (\ref{v2}) are satisfied for all $d \in \mathbb{N}$ with the same constants $\widetilde{C}$ and $\widetilde{D}$ and  $\Theta := \sup_{L,d} \left( \frac{m_{L,1}}{m_{L-d+1,1}}   \right)^{2\beta} \theta^2_d $ is bounded, then we may choose the constants $(\hat{c}_d)_{d \in \mathbb{N}}$ in such a way that they are all equal.
\end{lemma}
\begin{proof}
Choosing $f_l$ satisfying (\ref{v1}), due to exactness assumption we obtain
\begin{align*}
  \begin{split}
  A(L,d)f_l & = \sum_{\bl \in Q(L,d)} \bigotimes_{n = 1}^{d} \Delta^{(n)}_{l_n} f_{n,l}
 \\
  & = \sum_{t = 1}^{L-d+1} \Delta_t^{(1)}g_{1,l}\otimes \Delta_1^{(2)}g_{2,l} \cdots \otimes \Delta_1^{(d)}g_{d,l}
  \\
  & =  U^{(1)}_{L-d+1}g_{1,l} \otimes S^{(2)}g_{2,l} \cdots \otimes S^{(d)}g_{d,l}.
 \end{split}
\end{align*}
Let us put $T := \bigotimes_{n = 2}^d S^{(n)}, h = \bigotimes_{n = 2}^d g_{n,l}.$
 Due to (\ref{UnivLowBoundR}) we have
\begin{align*}
e^{{\rm{x}}}(A(L,d))^2 & \geq  \Expec \lVert (S_d - A(L,d))f_l \rVert^2 = \Expec \lVert (S^{(1)} \otimes T - U^{(1)}_{L -d+1} \otimes T)f_l \rVert^2
\\
& = \Expec \lVert (S^{(1)} - U^{(1)}_{L-d+1})g_{1,l} \otimes Th \rVert^2 = \lVert Th \rVert^2 \Expec \lVert (S^{(1)} - U^{(1)}_{L-d+1})g_{1,l} \rVert^2
\\
& = \theta^2_d\Expec \lVert (S^{(1)} - U^{(1)}_{L-d+1})g_{1,l} \rVert^2  \geq \widetilde{C}^2\theta^2_d \left( \kl(1-K^{-1})   \right)^{2\beta} \left( \frac{m_{L,1}}{m_{L-d+1,1}}   \right)^{2\beta} m_{L,1}^{-2\beta}.
 \end{align*}
\end{proof}
\begin{remark}
In particular constants $(\hat{c}_d)_{d \in \mathbb{N}}$ may be chosen all equal e.g. when (\ref{v1}) and (\ref{v2}) are satisfied for all $d \in \mathbb{N}$ with the same constants $\widetilde{C},\widetilde{D}$ and $\theta := \sup_{d \in \mathbb{N}} \theta_d < \infty.$
\end{remark}
\begin{lemma}\label{NBoundInM}
Let there exist constants  $1\leq \kl  \leq \ku, 1< K$ such that for all $n = 1,\ldots, d$ and $l \in \mathbb{N}$ (\ref{m_est_prot}) is satisfied.  Then there exists a constant $\widetilde{c}_d$ such that
$$\widetilde{c}_d m_{L,1} (\log(m_{L,1}))^{d-1} \leq N(L,d)$$
for all $L \geq d.$
 Moreover, if $\xi := \xi(d) := (K^d-1)^{\frac{1}{d}} > 1$ then there exists a constant $\widetilde{C}_d$ such that
$$N(L,d) \leq \widetilde{C}_d m_{L,1} (\log(m_{L,1}))^{(d-1)}$$
for all $L\geq d.$

\end{lemma}
\begin{proof}
First we prove the upper bound.

On the one hand, due to (\ref{m_est}) it holds
\begin{align*}
m_{L,1} \log(m_{L,1})^{d-1} & \geq \kl \frac{K^L-1}{K^L} \log(\kl(K^L - 1))^{d-1} K^L
\\
& \geq \kl \frac{K^d - 1}{K^d} \log\left(\left[(\kl(K^L-1 ))^{\frac{1}{L}}\right]^L\right)^{d-1} K^L
\\
& \geq \kl \frac{K^d - 1}{K^d} \log(\xi)^{d-1} L^{d-1}K^L,
\end{align*}
where we  used that the function
$[0, \infty) \ni x \mapsto (K^x - 1)^{\frac{1}{x}} $
is increasing.

On the other hand, according to Lemma \ref{NBound}
$$N(L,d) \leq \kud \frac{K}{K-1} \binom{L-1}{d-1}K^L \leq \kud \frac{K}{K-1} \frac{1}{(d-1)!} L^{d-1}K^L.$$
It follows that the constant
$$\widetilde{C}_d = \kud \frac{K}{K-1} \frac{1}{(d-1)!} \left[ \kl \frac{K^d - 1}{K^d} \log(\xi)^{d-1}  \right]^{-1} $$
does the job.

Now we prove the lower bound.

On the one hand due to (\ref{m_est}) we have
\begin{align*}
 m_{L,1} (\log(m_{L,1}))^{d-1} \leq \ku(K^L - 1) \left(\log(\ku(K^L-1))\right)^{d-1} \leq \ku \left(\log(\ku^{\frac{1}{L}}K)\right)^{d-1} K^L L^{d-1}.
\end{align*}
Noticing that $f: [d, \infty) \rightarrow \mathbb{R}, x \mapsto \frac{(x-1)\cdots(x-d+1)}{x^{d-1}}$ is increasing  we obtain
\begin{align*}
& \binom{L-1}{d-1} = \frac{(L-1) \cdots (L-d+1)}{L^{d-1}} \frac{L^{d-1}}{(d-1)!} \geq \frac{(d-1)!}{(d)^{d-1}}\frac{L^{d-1}}{(d-1)!}
\\
& = \frac{1}{(d)^{d-1}} L^{d-1}.
\end{align*}
On the other hand we have due to Lemma \ref{NBound}
\begin{align*}
N(L,d) & \geq \left( \frac{(K-1)\kl}{K} \right)^d K^L \binom{L-1}{d-1}
\\
& \geq \left( \frac{(K-1)\kl}{K} \right)^d \frac{1}{(d)^{d-1}} K^L L^{d-1}.
\end{align*}
It follows that the constant
$$\widetilde{c}_d =  \left(\frac{(K-1)\kl}{K} \right)^d \frac{1}{(d)^{d-1}}\left(\ku \left(\log(\ku^{\frac{1}{d}}K)\right)^{d-1}\right)^{-1}$$
satisfies
$$\widetilde{c}_d m_{L,1} (\log(m_{L,1}))^{(d-1)} \leq N(L,d).$$
\end{proof}

\begin{remark}
Note that the constants $\widetilde{C}_d$ and $\widetilde{c}_d$ from Lemma \ref{NBoundInM} fall superexponentially fast in $d.$
\end{remark}

\begin{corollary}\label{Cor_Lower_Bound}
Let ${\rm{x}} \in \{\rm{r}, \rm{w} \}.$
Let  (\ref{v1}) and (\ref{v2}) hold. Furthermore, let there exist constants  $1\leq \kl  \leq \ku, 1< K$ such that for all $n = 1,\ldots, d$ and $l \in \mathbb{N}$ (\ref{m_est_prot}) is satisfied. Moreover, assume that $m_{L,1} \geq 16$. Then there exists a constant $c_d$   such that given $\delta \in (0,1)$ there exists $ N(\delta)$ such that for every $N \geq N(\delta)$
$$e^{{\rm{x}}}(A(L,d)) \geq c_d \frac{(\log(N))^{(d-1 - \delta) \beta}}{N^{\beta}}$$
with $\beta = \frac{\log(\frac{1}{\widetilde{D}})}{\log(K)}$ and $N = N(L,d).$
\end{corollary}

\begin{proof}
Let $\widetilde{c}$ be such that for every $L \in \mathbb{N}$
$$\widetilde{c} m_{L,1} (\log(m_{L,1}))^{d-1} \leq N(L,d).$$
The existence of $\widetilde{c}$ is guaranteed by Lemma \ref{NBoundInM}.
We put $\widetilde{c}_0 := \min\{\widetilde{c}, 1\}.$

We would like to express the bound from Lemma \ref{LowBoundLem} in terms of the cardinality $N := N(L,d)$. To this end we want to find a function $g:\mathbb{R} \rightarrow \mathbb{R}$ of the form $g(x) = \frac{x^{\beta}}{(\log(x))^{\eta}}$ such that for large $m$
\begin{equation}\label{TranslFun}
 g(m (\log(m))^{d-1}) \geq  m^{\beta}
\end{equation}
implying
$$e^{{\rm{x}}}(A(L,d)) \geq \frac{\hat{c}_d}{m^{\beta}_{L,1}} \geq \frac{\hat{c}_d}{g( m_{L,1} (\log(m_{L,1}))^{d-1})}.$$
We rewrite (\ref{TranslFun}) as
$$\frac{m^{\beta} (\log(m))^{(d-1)\beta}}{(\log(m(\log(m))^{d-1}))^{\eta}} \geq m^{\beta}.$$
Hence (\ref{TranslFun}) holds if
$$\eta \leq \frac{ (d-1) \beta \log(\log(m))}{\log\left(\log(m) + (d-1) \log(\log(m))\right)}$$
and the expression on the right hand side  converges from below to $(d-1)\beta$ as $m$ goes to $\infty.$
To obtain
$$\frac{\hat{c}_d}{g(m_{L,1} \log(m_{L,1})^{d-1})} \geq \frac{\hat{c}_d}{g(\widetilde{c}_0^{-1} N)}$$
 it is sufficient to check that g is increasing on the interval $[m_{L,1} \log(m_{L,1})^{d-1}, \infty).$ Simple calculations reveal that $g$ is increasing on $[e^{\frac{\eta}{\beta}}, \infty) \supset [e^{d-1}, \infty).$
The final step is to notice that
$$e^{{\rm{x}}}(A(L,d)) \geq \frac{\hat{c}_d}{g(\widetilde{c}_0^{-1} N)} = \hat{c}_d \frac{\log(\widetilde{c}_0^{-1} N)^{\eta}}{(\widetilde{c}_0^{-1} N)^{\beta}} \geq \hat{c}_d\widetilde{c}_0^{\beta} \frac{\log(N)^{\eta}}{N^{\beta}}.$$
Putting $c_d := \hat{c}_d \widetilde{c}_0^{\beta}$ finishes the proof.
\end{proof}

\section{Application to Infinite-Dimensional Integration}\label{INF_DIM_INT}

In Theorem \ref{Theo_UB_PW} we provide a sharp result
on randomized infinite-dimensional integration on weighted reproducing kernel Hilbert spaces that parallels the sharp result on deterministic infinite-dimensional integration stated in \cite[Theorem~5.1]{GHHR17}.
Results from \cite{Gne13} and from \cite{PW11} in combination with Theorem~\ref{CardInfUpBound} rigorously establish the sharp randomized result in the special case where the weighted reproducing kernel Hilbert space is based on an anchored univariate kernel. With the help of the embedding tools provided in \cite{GHHR17} this result will be extended to general weighted reproducing kernel Hilbert spaces. Before we can state and prove Theorem  \ref{Theo_UB_PW} we
first have to introduce the setting, cf. \cite{GHHR17}.

For basic results about reproducing
kernels $K$ and the corresponding Hilbert spaces $H(K)$ we refer to \cite{Aro50}.
We denote the norm on $H(K)$ by $\|\cdot \|_K$ and the space of constant functions (on a given domain) by $H(1)$;
here $1$ denotes the constant kernel that only takes the function value
one.

\subsection{Assumptions}\label{subsec:assumptions}

Henceforth we assume that
\begin{enumerate}[label=(A\arabic*)]
\item\label{a1}
$H$ is a vector space of real-valued functions on a domain $D \neq
\emptyset$ with $H(1) \subsetneq H$
\end{enumerate}
and
\begin{enumerate}[label=(A\arabic*)]
\setcounter{enumi}{1}
\item\label{a2}
$\|\cdot\|_1$ and $\|\cdot\|_2$ are seminorms on $H$,
induced by symmetric bilinear forms $\langle\cdot,\cdot\rangle_1$ and
$\langle\cdot,\cdot\rangle_2$, such that
$\|1\|_1=1$ and $\|1\|_2=0$.
\end{enumerate}
Let
\begin{equation}\label{eq3}
\|f\|_H := \left(\|f\|_1^2 + \|f\|_2^2\right)^{1/2}
\hspace{3ex}\text{for $f \in H$.}
\end{equation}
Furthermore, we assume that
\begin{enumerate}[label=(A\arabic*)]
\setcounter{enumi}{2}
\item\label{a3}
$\|\cdot\|_H$ is a norm on $H$ that turns this space into a
reproducing kernel Hilbert space, and there exists a constant $c
\geq 1$ such that
\begin{equation}\label{eq2}
\|f\|_H \leq c\left(\left|\langle f,1\rangle_1\right|+\|f\|_2\right)
\hspace{3ex}\text{for all $f \in H$.}
\end{equation}
\end{enumerate}

Condition \eqref{eq2} is equivalent to the
fact that $\|\cdot\|_H$ and
$\left|\langle\cdot,1\rangle_1\right|+\|\cdot\|_2$
are equivalent norms on $H$.


Let us restate  Lemma~2.1 from  \cite{GHHR17}:

\begin{lemma}\label{lem10}
For each $\gamma >0$ there exists a uniquely determined reproducing kernel
$k_{\gamma}$ on $D\times D$ such that $H(1+k_{\gamma}) = H$ as vector spaces and
\begin{equation*}
\|f\|^2_{1+k_{\gamma}} = \|f\|_1^2 + \frac{1}{\gamma} \|f\|^2_2.
\end{equation*}
Moreover, the norms $\|\cdot \|_H$ and $\|\cdot\|_{1+k_{\gamma}}$ are equivalent and
$H(1) \cap H(k_{\gamma})=\{0\}$.
\end{lemma}

Note that for the special value $\gamma =1$ we have $\|\cdot\|_{1+k_1} = \|\cdot\|_H$.

The next example illustrates the assumptions and the statement of Lemma \ref{lem10}; for more information and a slight generalization see
\cite[Example~2.3]{GHHR17}.

\begin{exmp}\label{e33}
Let $D:=
[0,1)$ and  $r>1/2$.
The \emph{periodic Sobolev space} $K_r = K_r([0,1))$ (also known as \emph{Korobov space}) is the
Hilbert space of all
$f\in L^2([0,1])$ with finite norm
\[
 \|f\|_{r}^2 := |\hat{f}(0)|^2 + \sum_{h\in\Z \setminus \{0\}} | \hat{f}(h)|^2 h^{2r},
\]
where
$\hat{f}(h)
=\int_0^{1} f(t)e^{- 2\pi i h t}\,\mathrm dt$
is the $h$-th Fourier coefficient of $f$. The functions in $K_r$ are continuous and periodic.
It is easily checked that the reproducing kernel of $K_r$ is given by
\begin{equation}\label{kernel_korobov}
1+k_1(x,y) = 1 + \sum_{h\in\Z \setminus \{0\}} h^{-2r} e^{2\pi i h (x-y)},
\hspace{3ex}x,y \in [0,1).
\end{equation}
Consider the pair of seminorms on $K_r$ given by
\begin{equation*}
\|f\|_{1} = | \hat{f}(0)|
\hspace{3ex}\text{and}\hspace{3ex}
\|f\|_{2}^2 = \sum_{h\neq 0} | \hat{f}(h)|^2 \, h^{2r}.
\end{equation*}
The assumptions \ref{a1}, \ref{a2},
and \ref{a3} are easily verified.
For $\gamma >0$ we have $k_{\gamma} = \gamma \cdot k_{1}$.


\end{exmp}

Further examples of spaces that satisfy the assumptions \ref{a1}, \ref{a2},
and \ref{a3} are, for instance, the \emph{(non-periodic) Sobolev spaces} $W^{r,2}([0,1])$ of smoothness
$r\in\N$ endowed with either the
standard norm, the anchored norm or the ANOVA norm, see \cite[Example~2.1]{GHHR17}.

We now want to study weighted tensor product Hilbert spaces of multivariate functions, which implies that we have to consider \emph{product weights} as introduced in \cite{SW98}. More precisely, we
consider a sequence
$\bg=\left(\gamma_j\right)_{j\in \N}$ of positive
weights that satisfies
\begin{equation}\label{summable}
\sum_{j=1}^\infty \gamma_j < \infty.
\end{equation}
The decay of the weights is quantified by
\begin{equation*}
\decay(\bg) := \sup \Bigl(
\Bigl\{ p > 0 \, \Big|\,
\sum^\infty_{j=1} \gamma_j^{1/p} < \infty \Bigr\} \cup \{0\}
\Bigr);
\end{equation*}
due to \eqref{summable} we have $\decay(\bg) \ge 1$.
For each weight $\gamma_j$ let $k_{\gamma_j}$ be the kernel from Lemma \ref{lem10}.
With the help of the weights we can define
spaces of functions of finitely many variables.
For $d\in\N$
we define the reproducing kernel $K^{\bsgamma}_d$ on
$D^d\times D^d$ by
\begin{align}\label{eq100}
K^{\bsgamma}_d(\bx,\by) := \prod_{j= 1}^d
(1+k_{\gamma_j}(x_j,y_j)),
\hspace{3ex} \bx, \by\in D^{d}.
\end{align}
The reproducing kernel Hilbert space $H(K^{\bsgamma}_d)$
is the (Hilbert space) tensor product of the spaces $H(1+k_{\gamma_j})$.

Now we want to define a space of functions of infinitely many variables.
The natural domain for the counterpart of \eqref{eq100} for
infinitely many variables is given by
\begin{align}\label{eq101}
\X^{\bg} :=
\Bigl\{{\bf x} \in D^{\mathbb{N}} \,\Big|\, \prod_{j=1}^\infty (1+
k_{\gamma_j}(x_j,  x_j) )<\infty \Bigr\}.
\end{align}
Let $a,a_1,\dots,a_n\in D$ be arbitrary. Due to \cite[Lemma~2.2]{GHHR17} we have
$(a_1,\dots,a_n,a,a,\dots)\in \X^{\bg}$,
and in particular $\X^{\bg} \neq\emptyset$.
We define the reproducing kernel $K^{\bsgamma}_\infty$ on
$\X^{\bsgamma}\times \X^{\bsgamma}$ by
\begin{align}\label{eq102}
K^{\bsgamma}_\infty({\bf x},{\bf y}) :=
\prod_{j=1}^\infty (1+k_{\gamma_j}(x_j, y_j)),
\hspace{3ex}\bx,\by \in \X^{\bg}.
\end{align}
For a function $f\colon D^d \to \R$ we define
$\psi_d f\colon \X^\bg \to \R$ by
\begin{align}\label{g8}
\left( \psi_d f \right) (\bx)
=f(x_1,\dots,x_d)
\hspace{3ex}\text{for $\bx\in\X^\bg$.}
\end{align}
Due to \cite[Lemma~2.3]{GHHR17} $\psi_d$ is a linear isometry from
$H(K_d^{\bg})$ into $H(K^{\bg}_\infty)$,
and
\begin{equation}\label{dense_subspace}
\bigcup_{d\in \N}\psi_d (H(K_d^{\bg}))
\hspace{2ex}
\text{is a dense subspace of $H(K^{\bg}_\infty)$.}
\end{equation}

\subsection{The Integration Problem}

To obtain a well-defined integration problem
we assume that $\rho$ is a probability measure on $D$ implying
\begin{align*}
H \subseteq L^1(D,\rho).
\end{align*}
Let $\rho^d$ and $\rho^\N$ denote the corresponding product
measures on $D^d$ and $D^\N$, respectively.

Due to \cite[Lemma~3.1]{GHHR17} we have
for all  $d\in\N$ that
\begin{align*}
H(K_d^\bg)\subseteq L^1(D^d,\rho^{d}),
\end{align*}
and the respective embeddings $J_d$
from $H(K_d^\bg)$ into $L^1(D^d,\rho^{d})$
are continuous with
\begin{equation}\label{uniformly_bounded}
\sup_{d\in\N}\|J_d\|_{\rm op}
<\infty.
\end{equation}
Define the linear functional $I_d\colon H(K_d^\bg)\to \R$ by
\begin{align*}
I_d(f)
=\int_{D^d}f\,{\rm d}\rho^{d},
\hspace{3ex}\text{$f\in H(K_d^\bg)$.}
\end{align*}
Note that $\|I_d\|_{\rm op} \geq 1$, since $I_d(1) = 1$ and $\|1\|_{K^\bg_d}=1$.
Furthermore, $\|I_d\|_{\rm op} \leq \|J_d\|_{\rm op}$,
and therefore \eqref{uniformly_bounded} implies
\begin{align}\label{eq104}
1 \leq \sup_{d\in\N}\|I_d\|_{\rm op}
<\infty.
\end{align}
This yields the existence of a uniquely determined bounded linear functional
\begin{equation}\label{integration_functional}
       I_\infty \colon H(K^\bg_\infty)\to\R
\hspace{2ex}\text{such that}\hspace{2ex}
I_\infty (\psi_d f)
=I_d(f)
\hspace{2ex}\text{for all $f\in H(K^\bg_d)$,\,$d\in\N$,}
\end{equation}
cf. \cite[Lemma~3.2]{GHHR17}.

Note that every $f\in H(K^\bg_\infty)$ is measurable
with respect to the trace of the product $\sigma$-algebra
on $D^\N$. (This follows  from \eqref{dense_subspace},
\eqref{uniformly_bounded}, and the fact that the pointwise limit
of measurable functions is again measurable.)

If $\X^\bg$ is measurable, $\rho^\N(\X^\bg)=1$, and
$H(K^\bg_\infty)\subseteq L^1(\X^\bg, \rho^\N)$, then
the bounded linear functional
\eqref{integration_functional} is given by
\begin{align*}
I_\infty (f)
=\int_{\X^\bg} f\,\mathrm d\rho^\N
\hspace{3ex}\text{for all $f\in H(K^\bg_\infty)$.}
\end{align*}
For sufficient conditions under which these assumptions are fulfilled
we refer to \cite{GMR12}.

We consider the integration problem on $H(K^\bg_\infty)$ consisting in the
approximation of the functional $I_\infty$ by randomized
algorithms that use function evaluations (i.e., \emph{standard information}) as admissable
information.

\subsection{The Unrestricted Subspace Sampling Model}

We use the cost model introduced in \cite{KSWW10a}, which we refer to as \emph{unrestricted subspace sampling model}.
It only accounts for the cost of function evaluations. To define the cost of a function evaluation, we fix an anchor $a \in D$ and a
non-decreasing function
$$\$: \mathbb{N}_0 \to [1,\infty].$$
Put
\begin{equation*}
\mathcal{U}:= \{u\subset \N\,|\, |u| <\infty\}.
\end{equation*}
For each $u\in\mathcal{U}$ put
\begin{equation*}
 \mathcal{T}_{u} := \{ {\bf t} \in D^\mathbb{N} \,|\, t_j = a \hspace{1ex}\text{for all}
\hspace{1ex} j\in\mathbb{N}\setminus u\}.
\end{equation*}
To simplify the representation, we confine ourselves to non-adaptive randomized linear algorithms of the form
\begin{equation}\label{Qn}
Q(f) = \sum^n_{i=1} w_i f(\bst^{(i)}),
\end{equation}
where the number $n \in \N$ of
knots is fixed and the knots $\bst^{(i)}$ as well as the coefficients
$w_i\in \R$ are random variables with values in some $\mathcal{T}_{v_i}$, $v_i\in \mathcal{U}$, and in $\R$,
respectively. (We discuss a larger class of algorithms in Remark~\ref{Baustelle}.)
The cost of $Q$ is given by
\begin{align}\label{cost_function}
\cost(Q)
=\sum_{i=1}^n
\inf\{ \$(|u|) \mid
\text{$u \in \mathcal{U}$ such that
$\bst^{(i)}(\omega) \in \mathcal{T}_u$ for all $\omega \in \Omega$} \}.
\end{align}
In the definition of the cost function an inclusion property has to hold for all $\omega \in \Omega$. Often this worst case point of view is replaced by an average case (cf., e.g., \cite{MGR09} or \cite{DG13, Gne13, PW11}). We stress that such a replacement would not affect the cost of the algorithms that we employ to establish our upper bounds for the $N$-th minimal errors; for lower bounds cf. Remark~\ref{Baustelle}(iii).

Let $\rm{x} \in \{\rm{d}, \rm{r}, \rm{w}\}$. For $N\ge 0$  let us define the  \emph{$N$-th minimal error on} $H(K^\bg_\infty)$ by
\begin{equation*}
 e^{\rm x}(N,K^\bg_\infty) := \inf\{ e^{\rm x}(I_\infty, Q) \,|\, Q\hspace{1ex}\text{as in \eqref{Qn}}
\hspace{1ex}\text{and}\hspace{1ex}{\rm cost}_{}(Q) \le N\},
\end{equation*}
where in the case $\rm{x} = \rm{d}$ the algorithms have to be deterministic, while
in the case $\rm{x} \in \{\rm{r}, \rm{w}\}$ they are allowed to be randomized.
The \emph{(polynomial) convergence order of the $N$-th minimal errors of infinite-dimensional integration} is given by
\begin{align}\label{lambdacost}
\lambda^{{\rm x}}(K^\bg_\infty)
:=\sup \Bigl\{ \alpha \ge 0 \mid \sup_{N\in \N}
e^{\rm x}(N,K^\bg_\infty) \cdot  N^\alpha < \infty
\Bigr\}.
\end{align}

In analogy to our definitions for infinite-dimensional integration, we consider for univariate integration on $H(1+k_1)$ also
linear randomized algorithms $Q$ of the form \eqref{Qn}, except that this time the knots $\bst^{(i)}$ are, of course, random variables with values in $D$. The cost of such an algorithm is simply the number $n$ of function evaluations, and \emph{$N$-th minimal errors
on} $H(1+k_1)$ are
given by
\begin{equation*}
 e^{\rm x}(N,1+k_1) := \inf\{ e^{\rm x}(I_1, Q) \,|\, Q\hspace{1ex}\text{as in \eqref{Qn}}
\hspace{1ex}\text{and}\hspace{1ex}n \le N\}.
\end{equation*}
The \emph{(polynomial) convergence order of the $N$-th minimal errors of univariate integration} is given by
\begin{align*}
\lambda^{{\rm x}}(1+k_1)
:=\sup \Bigl\{ \alpha \ge 0 \mid \sup_{N\in \N}
e^{\rm x}(N, 1+k_1) \cdot  N^\alpha < \infty
\Bigr\}.
\end{align*}

\begin{remark}\label{Rem_Rel_Min_Err}
Let $K\in \{K^{\bg}_\infty, 1 + k_1\}$
and, accordingly, $I\in \{I_\infty, I_1\}$. Obviously,
\begin{equation*}
e^{\rm r}(N, K) \le e^{\rm w}(N, K) \le e^{\rm d}(N, K),
\hspace{3ex}\text{and thus}\hspace{3ex}
\lambda^{\rm r}(K) \ge \lambda^{\rm w}(K) \ge \lambda^{\rm d}(K).
\end{equation*}
Furthermore, it is easy to see that  $e^{\rm w}(N,K) \ge e^{\rm d}(N,K)$ holds:
If $Q$ is an arbitrary randomized algorithm of the form \eqref{Qn} with $\cost(Q) \le N$, then for every $\omega \in \Omega$ the cost of the deterministic algorithm $Q(\omega)$  is at most $N$, implying
$$
\sup_{\|f\|_K \le 1} |(I-Q(\omega))f| \ge e^{\rm d}(N,K),
$$
which in turn leads to $e^{\rm w}(I,Q) \ge e^{\rm d}(N,K)$.
Hence we obtain
\begin{equation}\label{d=w}
e^{\rm w}(N,K) = e^{\rm d}(N,K)
\hspace{3ex}\text{and}\hspace{3ex}
\lambda^{\rm w}(K) = \lambda^{\rm d}(K).
\end{equation}
\end{remark}

\subsection{A Sharp Result on Infinite-Dimensional Integration}

The next theorem determines the exact polynomial convergence rate of the $N$-th minimal
errors of infinte-dimensional integration on weighted reproducing kernel Hilbert spaces.

\begin{theorem}\label{Theo_UB_PW}
Let $\rm{x} \in \{\rm{r}, \rm{w}\}$.
If the cost function $\$$ satisfies $\$(\nu) = \Omega(\nu)$ and
$\$(\nu) = O(e^{\sigma \nu})$ for some $\sigma\in (0,\infty)$,
then we have
\begin{equation}\label{sharp_bound}
\lambda^{{\rm x}}(K^{\bg}_\infty) =
\min \left\{ \lambda^{{\rm x}}(1+k_1),\,
\frac{{\rm decay}(\bg) - 1}{2} \right\}.
\end{equation}
\end{theorem}

Notice that the theorem implies that in the randomized setting infinite-dimensional integration on weighted reproducing kernel Hilbert spaces is (essentially) not harder than the corresponding univariate integration problem (as far as the polynomial convergence rate is concerned) as long as the weights decay fastly enough, i.e., as long as
$${\rm decay}(\bg) \ge 2  \lambda^{{\rm x}}(1+k_1) +1.$$

\begin{proof}
Let us first consider the case ${\rm x} = {\rm r}$.
In the special case where the reproducing kernel $k_1$ is \emph{anchored in $a$} (i.e., $k_1(a,a)=0$) and satisfies $\gamma k_1 = k_\gamma$ for all $\gamma >0$ (cf. Lemma~\ref{lem10}), the statement of the Theorem follows from \cite{Gne13} and from \cite{PW11} in combination with Theorem~\ref{CardInfUpBound}, as we will explain below in detail.

For a general reproducing kernel $k_1$ we need to find a suitably associated reproducing kernel $k_a$ anchored in $a$ and satisfying $\gamma k_a = (k_a)_\gamma$ for all $\gamma >0$ to employ the embedding machinery from \cite{GHHR17} to obtain the desired result \eqref{sharp_bound}. To this purpose we
consider the bounded linear
functional $\xi: H\to \R$,  $f \mapsto f(a)$, where $a\in D$ is our fixed anchor.
We define a new pair of seminorms on $H$ by
\begin{equation*}
\|f\|_{1,a} := |\xi(f)|
\hspace{3ex}\text{and}\hspace{3ex}
\|f\|_{2,a} := \|f-\xi(f)\|_{H}.
\end{equation*}
Notice that $\|\cdot \|_{1,a}$ is induced by the symmetric bilinear form
$\langle f,g \rangle_{1,a} := \xi(f) \cdot \xi(g)$.
This new pair of seminorms satisfies obviously assumption  \ref{a2} and
the norms $\|\cdot \|_{H} = ( \|\cdot\|^2_1 + \|\cdot\|^2_2)^{1/2}$ and
$\|\cdot \|_{H,a} := ( \|\cdot\|^2_{1,a} + \|\cdot\|^2_{2,a})^{1/2}$
are equivalent norms on $H$. Hence $\|\cdot\|_{H,a}$ turns $H$ into a reproducing kernel Hilbert space, and satisfies \eqref{eq2}
with $c=1$ since
\begin{equation*}
\|f\|_{H,a} \le \| f \|_{1,a} + \|f\|_{2,a}
= | \langle f, 1 \rangle_{1,a}| + \|f\|_{2,a}
\hspace{3ex}\text{for all $f\in H$.}
\end{equation*}
Thus the new pair of seminorms satisfies also \ref{a3}.
Furthermore, if $k_a$ is the reproducing kernel on $D\times D$ such that
\begin{equation*}
H(k_a) = \{f\in H\,|\, f(a)=0\}
\end{equation*}
and
\begin{equation*}
\|f\|_{k_a} = \|f\|_{1+ k_a} = \|f\|_{H,a}
\hspace{3ex}\text{for all $f\in H(k_a)$,}
\end{equation*}
then $k_a$ is anchored in $a$ and moreover we have
$H(1+k_a) = H$ as vector spaces, $H(1) \cap H(k_a) = \{0\}$, and
\begin{equation*}
\|f\|^2_{1+\gamma k_a} = \|f\|^2_{1,a} + \frac{1}{\gamma} \|f\|^2_{2,a}
\end{equation*}
for all $\gamma >0$, $f\in H$, implying $(k_a)_\gamma = \gamma k_a$, see \cite[Rem.~2.2]{GHHR17}.
Since $ \|\cdot\|_H = \|\cdot\|_{1+k_1}$ and $\|\cdot \|_{H,a} = \|\cdot \|_{1+k_a}$ are equivalent norms on $H(1+k_1) = H = H(1+k_a)$, we obtain $\lambda^{{\rm r}}(1+k_1) =
\lambda^{{\rm r}}(1+k_{a})$.
Due to \cite[Thm.~2.3]{GHHR17} we have
\begin{align*}
\X^{\bg,a} :=
\Bigl\{{\bf x} \in D^{\mathbb{N}} \,\Big|\, \prod_{j=1}^\infty (1+
\gamma_j k_{a}(x_j,  x_j) )<\infty \Bigr\} = \X^{\bg}.
\end{align*}
According to \eqref{eq102}
we define $K^{\bsgamma,a}_\infty: \X^{\bg} \times \X^{\bg} \to \R$ by
\begin{equation*}
K^{\bsgamma,a}_\infty({\bf x},{\bf y}) :=
\prod_{j=1}^\infty (1+\gamma_j k_{a}(x_j, y_j))
\hspace{3ex}\text{for $\bx,\by \in \X^{\bg}$.}
\end{equation*}
Now we consider the integration problem in $H(K^{\bsgamma,a}_\infty)$ and may use \cite[Subsect.~3.2.1]{Gne13} and \cite[Cor.~1]{PW11} in combination with Theorem \ref{CardInfUpBound}.
Indeed, due to Theorem \ref{CardInfUpBound}
we may choose linear randomized algorithms with convergence rates $\alpha$ arbitrarily
close to $\lambda^{{\rm r}}(1+k_{1}) = \lambda^{{\rm r}}(1+k_{a}) $ to obtain via the randomized Smolyak method algorithms that
satisfy \eqref{pw_req_est}  for ${\rm x} = {\rm r}$
(and consequently also \cite[Eqn.~(10)]{PW11}).
Now \cite[Cor.~1]{PW11} ensures that
\begin{equation*}
\lambda^{{\rm r}}(K^{\bg, a}_\infty) \ge
\min \left\{ \lambda^{{\rm r}}(1+k_1),\,
\frac{{\rm decay}(\bg) - 1}{2} \right\}.
\end{equation*}
Furthermore, we have due to \cite[Eqn.~(21)]{Gne13}
\begin{equation*}
\lambda^{{\rm r}}(K^{\bg,a}_\infty) \le
\min \left\{ \lambda^{{\rm r}}(1+k_1),\,
\frac{{\rm decay}(\bg) - 1}{2} \right\}.
\end{equation*}
Due to \cite[Cor.~5.1]{GHHR17} these estimates also hold for
$H(K^{\bg}_\infty)$.

Let us now consider the case ${\rm x} = {\rm w}$.
Due to \eqref{d=w}, identity \eqref{sharp_bound} follows
directly from the deterministic result \cite[Theorem~5.1]{GHHR17}.
\end{proof}

We now provide two corollaries and to add some remarks.

Theorem~\ref{Theo_UB_PW}, which deals with randomized algorithms, and the corresponding deterministic theorem \cite[Theorem~5.1]{GHHR17} allow immediately to compare the power of deterministic and randomized algorithms.

\begin{corollary}\label{Power}
Let the assumptions of Theorem~\ref{Theo_UB_PW} hold. For infinite-dimensional integration on $H(K^{\bg}_\infty)$
randomized algorithms are superior to deterministic algorithms, i.e., $\lambda^{{\rm r}}(K^{\bg}_\infty) > \lambda^{{\rm d}}(K^{\bg}_\infty)$, if and only if
\begin{equation*}
\lambda^{{\rm r}}(1+k_1) > \lambda^{{\rm d}}(1+k_1)
\hspace{3ex}\text{and}\hspace{3ex}
\decay(\bg) > 1 + 2  \lambda^{{\rm d}}(1+k_1)
\end{equation*}
are satisfied.
\end{corollary}

The next corollary on infinite-dimensional integration on weighted Korobov spaces in the randomized setting parallels
\cite[Theorem~5.5]{GHHR17}, which discusses the deterministic setting.


\begin{corollary}\label{Korobov}
Let $r>1/2$, and let the univariate reproducing kernel $k_1$ be as in \eqref{kernel_korobov}. Then the  weighted Korobov space $H(K^\bg_\infty)$ is an infinite tensor product of the
periodic Korobov space $H(1+k_1) = K_r([0,1))$ of smoothness $r$, see
Example \ref{e33}.
If the cost function $\$$ satisfies $\$(\nu) = \Omega(\nu)$ and
$\$(\nu) = O(e^{\sigma \nu})$ for some $\sigma\in (0,\infty)$,
then we have
\begin{equation*}
\lambda^{{\rm r}}(K^{\bg}_\infty) =
\min \left\{ r + \frac{1}{2},\, \frac{{\rm decay}(\bg) - 1}{2} \right\}
\hspace{3ex}\text{and}\hspace{3ex}
\lambda^{{\rm w}}(K^{\bg}_\infty) =
\min \left\{ r,\, \frac{{\rm decay}(\bg) - 1}{2} \right\}.
\end{equation*}
\end{corollary}

\begin{proof}
Since $\lambda^{{\rm r}}(1+k_1) = r+1/2$ and $\lambda^{{\rm w}}(1+k_1) = r$ (see
Appendix and Remark~\ref{Rem_Rel_Min_Err}),
Theorem \ref{Theo_UB_PW} immediately yields the result for $\lambda^{{\rm r}}(K^{\bg}_\infty)$ and $\lambda^{{\rm w}}(K^{\bg}_\infty)$.

Notice that the result for $\lambda^{{\rm w}}(K^{\bg}_\infty)$ can also be derived from
Remark~\ref{Rem_Rel_Min_Err} and \cite[Theorem~5.5]{GHHR17}.
\end{proof}

\begin{remark}\label{Baustelle}
Let us come back to Theorem~\ref{Theo_UB_PW}.

\nopagebreak

\begin{itemize}
\item[(i)] Algorithms that achieve convergence rates arbitrarily close to $\lambda^{\rm x}(K^{\bg}_\infty)$
are, e.g., \emph{multivariate decomposition methods (MDMs)} that were introduced in \cite{KSWW10a} (in the deterministic setting) and developed further in \cite{PW11} (in the deterministic and in the randomized setting); originally, these algorithms were called \emph{changing dimension algorithms}, cf., e.g., \cite{DG12, DG13, Gne13, KSWW10a, PW11}. MDMs exploit that the anchored function decomposition of an integrand can be efficiently computed; a method for multivariate integration based on the same idea is the \emph{dimension-wise integration method} proposed in \cite{GH10}.
To achieve (nearly) optimal convergence rates, the MDMs may employ as building blocks Smolyak algorithms for multivariate integration that rely on (nearly) optimal algorithms for univariate integration on $H(1+k_1)$, cf. \cite[Section~3.3]{PW11} and the proof of Theorem~\ref{Theo_UB_PW}.
\item[(ii)] In the special case where ${\rm x} = {\rm r}$ and where $k_1$ is an \emph{ANOVA-kernel} (i.e., $k_1$ satisfies $\int_D k_1(y,x) \,{\rm d}x = 0$ for every $y\in D$) a
version of Theorem~\ref{Theo_UB_PW} was already proved in \cite[Theorem~4.3]{DG13}.
It was the first result that rigorously showed that MDMs can achieve the optimal order of convergence also on spaces with norms that are \emph{not} induced by an underlying anchored function space decomposition. It was not derived with the help of function space embeddings, but by an elaborate direct analysis. Apart from addressing only the ANOVA setting, a further drawback of \cite[Theorem~4.3]{DG13} is that its
assumptions are slightly stronger than the ones made in Theorem~\ref{Theo_UB_PW}:
It is not sufficient to know the convergence rate of the $N$-th minimal errors of the univariate integration problem, but additionally one has to verify the existence of unbiased randomized algorithms for multivariate integration that satisfy certain variance bounds, see \cite[Assumption~4.1]{DG13}. Nevertheless, in many important cases it is well known that such variance bounds hold. Furthermore, one should mention that the analysis in \cite{DG13} ist not restricted to product weights as in this section, but is done for general weights.

Note that the kernel $k_1$ of the Korobov space $K_r([0,1))$ from Example~\ref{e33} and Corollary~\ref{Korobov} is actually an ANOVA kernel. Hence the identity for $\lambda^{{\rm r}}(K^{\bg}_\infty)$ in Corollary~\ref{Korobov} may also be derived by employing \cite[Theorem~4.3]{DG13} after verifying the existence of unbiased algorithms for multivariate integration that satisfy \cite[Assumption~4.1]{DG13}.

\item[(iii)]  The upper bound for $\lambda^{{\rm r}}(K^{\bg}_\infty)$ in
\eqref{sharp_bound} relies on the corresponding bound \cite[Eqn.~(21)]{Gne13}
for the case where the univariate reproducing kernel $k_1$ is anchored in $a$.
Although the definition of the cost function in \cite{Gne13} takes the average case and not the worst case point of view and differs therefore from \eqref{cost_function}, both definitions lead to the same cost for the admissable class of algorithms $\mathcal{A}^{\rm res}$  considered in the unrestricted subspace sampling model in \cite{Gne13}. The class
$\mathcal{A}^{\rm res}$ contains not only algorithms of the form \eqref{Qn}, but also adaptive and non-linear algorithms.
In the proof of Theorem~\ref{Theo_UB_PW} we employ the function space embeddings from \cite{GHHR17}, which allows us to transfer results for linear algorithms from the case of anchored kernels to the general case. Hence we can conclude that the upper bound for $\lambda^{{\rm r}}(K^{\bg}_\infty)$ in \eqref{sharp_bound} holds also if we admit adaptive linear algorithms of the form  \eqref{Qn} for infinite-dimensional as well as for univariate integration, but we do not know whether this is still the case if we admit non-linear algorithms.
\end{itemize}
\end{remark}

We finish this section with some remarks on extensions of our results on infinte-dimensional integration to other settings.

\begin{remark}\label{Generalizations}
To obtain computational tractability of problems depending on a high or infinite number of variables, it is
usually essential to be able to arrange the variables in such a way that their impact decays sufficiently fast. One approach to model the decreasing impact of successive variables is to use weighted function spaces, like the ones we defined and studied in this section, to moderate the influence of groups of variables. This approach goes back to the seminal paper \cite{SW98}. Another approach is the concept of \emph{increasing smoothness} with respect to properly ordered variables,
see, e.g., \cite{DG16, GHHRW18, HHPS18, IKPW16, KPW14a, PW10, Sie14}. The precise definition of Hilbert spaces of functions depending on infinitely many variables of increasing smoothness can be found in \cite[Section~3]{GHHRW18}. Now \cite[Theorem~3.19]{GHHRW18} shows how to relate these spaces to suitable weighted Hilbert spaces via  mutual embeddings,
making it therefore easy to transfer our results  in the randomized setting, Theorem \ref{Theo_UB_PW} and Corollary \ref{Korobov}, from weighted spaces to spaces with increasing smoothness, cf. \cite[Theorem~4.5 and Corollary~4.7]{GHHRW18} for the corresponding transference results in the deterministic setting.

Instead of applying our result Theorem \ref{CardInfUpBound}
to the infinite-dimensional integration problem, we may also use it to tackle the \emph{infinite-dimensional $L_2$-approximation problem}. Indeed, a sharp result for the latter problem was obtained in \cite[Corollary~9]{Was12} in the deterministic setting
for weighted anchored reproducing kernel Hilbert spaces  with the help of multivariate decomposition methods based on Smolyak algorithms (cf. \cite[Theorem~7]{WW11b}).
The analysis relies on explicit cost bounds for deterministic Smolyak algorithms from \cite{WW95}. In \cite[Theorem~4.5]{GHHRW18} the result is extended to weighted (not necessarily anchored) reproducing kernel Hilbert spaces  (relying on the embedding tools from \cite{GHHR17}) and to spaces of increasing smoothness.

Now one may use Theorem \ref{CardInfUpBound} to establish a corresponding result
to \cite[Corollary~9]{Was12} for weighted anchored spaces in the randomized setting
and may generalize it to non-anchored weighted spaces and to spaces of increasing smoothness via the embedding results established in \cite{GHHR17, GHHRW18}.

To work out all the details of these generalizations
is beyond the scope of the present paper.
\end{remark}

\section{Appendix}\label{app}
\subsection{Randomized Integration Error in Korobov Spaces}

For $r > \frac{1}{2}$ we denote via $K_{r} := K_{r}([0,1))$ the space of Korobov functions on a one-dimensional torus with smoothness parameter $r.$ The space is equipped with the norm
$$\lVert f \rVert^2_r := |\hat{f}(0)|^2 + \sum_{h \in \mathbb{Z} \setminus \{0\}} |\hat{f}(h)|^2 |h|^{2r},$$
see Example \ref{e33}.
It is a folklore result that the polynomial convergence order $\lambda^{\rm{r}}(1+k_1)$
of randomized quadratures on $K_r$ is equal to $ r + \frac{1}{2}.$
Since we have not found in the literature a complete proof handling all cases $r > \frac{1}{2}$, we decided to provide a proof sketch in this appendix.
Similar reasoning for (non-periodic) Sobolev spaces with integer parameter $r$ may be found, e.g., in \cite[Chapter ~2.2]{Nov88}.
For the lower error bound we need the following lemma.

\begin{lemma}\label{LowerBoundLemma}
Let $r > \frac{1}{2}.$ There exists a sequence of functions $(f_n)_{n \in \mathbb{N}} \in K_r$ and a constant $C > 0$ such that for every $n \in \mathbb{N}$
\begin{enumerate}
\item $\supp(f_n) \subset (0, n^{-1})$,
\item $\int_{[0,1)} f_n \, dx = n^{-r-1},$
\item $\lVert f_n \rVert^2_r \leq Cn^{-1}.$ \label{NormCondition}
\end{enumerate}
\end{lemma}

\begin{proof}
Let $f$ be a positive infinitely many times differentiable function with $ \supp(f) \subset (0,1)$ and integral equal to $1.$ Define
$$f_n(x) := n^{-r} f(nx), n \in \mathbb{N}.$$ The sequence $(f_n)_{n \in \mathbb{N}}$ is the sequence we are looking for. Since all other properties are obvious it is enough to check that for $(f_n)_n, n \in \mathbb{N}$ the condition (\ref{NormCondition}) holds. Fix $n.$ Since $f$ (as a bump function) is in the Schwartz space $\mathcal{S}(\mathbb{R}),$ the same holds for its Fourier transform $\hat{f}$ and as a result there exists $\tilde{C}>0$ such that for every $x \in \mathbb{R}$
$$|\hat{f}(x)|^2 \leq \left( \frac{\tilde{C}}{(1+|x|)^{r+1}}  \right)^2.$$
Since $|\hat{f_n}(0)|^2 \leq n^{-3}$ we may neglect it. Simple calculations reveal
\begin{align*}
& \sum_{h \in \mathbb{Z} \setminus \{0\}} |\hat{f_n}(h)|^2 |h|^{2r} = \frac{1}{n}  \sum_{h \in \mathbb{Z} \setminus \{0\}} \frac{1}{n} |\hat{f}(\frac{h}{n})|^2 \left(\frac{|h|}{n}\right)^{2r}
\\
& \leq \frac{1}{n}  \sum_{h \in \mathbb{Z} \setminus \{0\}} \frac{1}{n}   \left( \frac{\tilde{C}}{(1+\frac{|h|}{n})^{r+1}}  \right)^2         \left(\frac{|h|}{n}\right)^{2r} =: \frac{1}{n} \rho_n.
\end{align*}
Due to monotonicity of $\mathbb{N} \ni h \mapsto \frac{1}{n} \left( \frac{\tilde{C}}{(1 + \frac{h}{n})^{r+1}}  \right)^2\left( \frac{h}{n} \right)^{2r}$ we obtain
\begin{align*}
& \rho_n \leq 2 \frac{1}{n} \left( \frac{\tilde{C}}{(1 +\frac{1}{n})^{r+1}}   \right)^2\left( \frac{1}{n} \right)^{2r} + 2 \int_{0}^{\infty} \frac{1}{n} \left( \frac{\tilde{C}}{(1 + \frac{t}{n})^{r+1}}  \right)^2\left( \frac{t}{n} \right)^{2r} \, dt
\\
& \leq 2\tilde{C}^2\left( 1 + \int_{0}^{\infty} \frac{s^{2r}}{(1 + s)^{2r + 2}}  \, ds    \right),
\end{align*}
which finishes the proof.
\end{proof}

For the upper bound we need a lemma on the $L^2-$approximation of functions from Korobov spaces by trigonometric polynomials.

\begin{lemma}\label{UpperBoundLemma}
Let $r > \frac{1}{2}, f \in K_r, N \in \mathbb{N}$, and let $q_N$ be the trigonometric polynomial
defined by
\begin{equation*}
q_N(x) := \sum_{k=1-N}^N \widehat{\alpha}_k e^{2\pi i kx},
\hspace{3ex} x\in [0,1),
\end{equation*}
where
\begin{equation*}
 \widehat{\alpha}_k := \frac{1}{2N} \sum^{2N-1}_{j=0} f \left( \tfrac{j}{2N} \right)  e^{-\pi i kj/N}
 \hspace{3ex}\text{for $k=1-N, 2-N, \ldots, 0 ,1,\ldots,N$,}
\end{equation*}
are the $2N$ discrete Fourier coefficients of $f$.
Then we have
\begin{equation*}
f \left( \tfrac{j}{2N} \right) = q_N \left( \tfrac{j}{2N} \right)
\hspace{3ex}\text{for $j=0,1,\ldots,2N-1$,}
\end{equation*}
and
\begin{equation*}
\lVert f - q_N \rVert^2_{L^2([0,1))} \leq (1+C_r) N^{-2r} \lVert f \rVert^2_r,
\hspace{3ex}\text{where $C_r := 2\sum_{\ell=1}^\infty (2\ell -1)^{-2r}$.}
\end{equation*}
The discrete Fourier coefficients $\widehat{\alpha}_k$, $k=1-N, 2-N, \ldots, 0 ,1,\ldots,N$, can be computed via the fast Fourier transform at cost $O(N \ln(N))$.
\end{lemma}

\begin{proof}
Proofs of the statements of the lemma can be found in many standard texts on numerical analysis.
We follow the course of \cite[Sections~52 and 53]{Han02}.
Since $r > \frac{1}{2}$ we may write $f$ as a uniformly convergent Fourier series
$$f(x) = \sum_{h \in \mathbb{Z}} \hat{f}(h)e^{2 \pi i h x}.$$
It is well-known (and not difficult to calculate) that if $q_N$ is a trigonometric polynomial of degree $N$ interpolating $f$ in the nodes $\frac{j}{2N}, j=0,1,\ldots, 2N-1$ then it is given by
$$q_N(x)  =  \sum_{j = 1-N}^{N} \left( \sum_{\ell \in \mathbb{Z}}  \hat{f}(j + 2\ell N) \right) e^{2 \pi i j x},$$
see for example Lemma $52.5$ in \cite{Han02}. It holds
\begin{equation*}
\begin{split}
\|f - q_N \|^2_{L^2([0,1))} &= \sum_{j = 1-N}^{N} \left|  \widehat{\alpha}_j - \hat{f}(j) \right|^2
+ \sum_{j = - \infty}^{-N} \left| \hat{f}(j) \right|^2 + \sum_{j = N+1}^{\infty} \left| \hat{f}(j) \right|^2 \\
&\le \sum_{j = 1-N}^{N} \Bigg| \sum_{|\ell |\geq 1}  \hat{f}(j + 2\ell N) \Bigg|^2
+ \sum_{|j| \ge N} \left| \hat{f}(j) \right|^2.
\end{split}
\end{equation*}
The last sum may be bounded in an obvious way by $N^{-2r}\lVert f \rVert^2_r$ and for the double sum we may use the Cauchy Schwarz inequality to obtain
\begin{equation*}
\begin{split}
\sum_{j = 1-N}^N \Bigg| \sum_{|\ell | \geq 1} \hat{f}(j + 2\ell N) \Bigg|^2
&= \sum_{j = 1-N}^N \Bigg| \sum_{|\ell | \geq 1} \frac{1}{(j+2\ell N)^r} \hat{f}(j+2\ell N) (j+2\ell N)^r  \Bigg|^2
\\
&\leq \sum_{j = 1-N}^N \left( \sum_{|\ell | \geq 1} (j + 2\ell N)^{-2r} \right)
\left( \sum_{|\ell | \geq 1} \left| \hat{f}(j + 2\ell N) \right|^2 (j+2\ell N)^{2r} \right)
\\
&\leq \left( \sum_{|\ell | \geq 1} (2\ell N - N)^{-2r} \right)
\left( \sum_{k \in \Z} \left| \hat{f}(k) \right|^2 k^{2r} \right)\\
&\leq N^{-2r} \left( 2\sum_{\ell=1}^\infty (2\ell -1)^{-2r} \right) \|f\|^2_r.
\end{split}
\end{equation*}
The cost analysis of the fast Fourier transform is well known and can, e.g., be found in \cite[Section~53]{Han02}.
\end{proof}

\begin{theorem}
Let $r > \frac{1}{2}.$ It holds
$$\lambda^{\rm{r}}(1+k_1) = r + \frac{1}{2}.$$
\end{theorem}
\begin{proof}
The upper bound on $\lambda^{\rm{r}}(1+k_1)$ is settled immediately by Lemma \ref{LowerBoundLemma} in conjunction with Corollary $7.35$ from \cite{MGNR12}. Indeed,
for $N\in \N$ choose $n:= 6N$ and
$$g_k(x) = f_{n} \big( x - \tfrac{k-1}{n} \big), \hspace{3ex} k=1,\ldots, n.$$
Then the assumptions of
Corollary $7.35$ from \cite{MGNR12} are satisfied for $\epsilon = n^{-r-1}$
implying $e^{\rm r}(N, 1+k_1) \ge \tfrac{1}{6^{r+1}\sqrt{2}}N^{-r-1/2}$, and thus
establishing the upper bound.

To get the lower bound on $\lambda^{\rm{r}}(1+k_1)$ consider the following algorithm. Let $N \in \mathbb{N}.$ We first interpolate $f \in K_r$ with the trigonometric polynomial $q_N$ from Lemma~\ref{UpperBoundLemma}. Now the integral over $q_N$ is simply the discrete Fourier coefficient $\widehat{\alpha}_0$.
To approximate the integral of $g:=(f - q_N)$ we apply a simple Monte Carlo quadrature. To this end let $(X_j)_{j = 0}^{N-1}$ be independent random variables such that $X_j$ is distributed uniformly on $[0,1)$. We put
$$Qg = \frac{1}{N}\sum_{j = 0}^{N-1} g(X_{j}).$$
Now, since $Qg$ is unbiased and $(X_j)_{j = 0}^{N-1}$ are independent, applying Lemma \ref{UpperBoundLemma} we get
\begin{equation}
 \Expec \left[  (Qg - \int_{[0,1)} g \, dx)^2 \right]^{\frac{1}{2}}  \leq \frac{1}{N^{\frac{1}{2}}} \lVert g \rVert_{L^2([0,1))}  \leq \frac{\sqrt{1+C_r}}{N^{\frac{1}{2}} N^{r}} \lVert f \rVert_r.
\end{equation}
Since cost of the algorithm is of order  $O(N\ln(N))$, the claim follows.
\end{proof}

\subsection*{Acknowledgment}

The authors thank Stefan Heinrich for pointing out the reference \cite{HM11}.
Part of the work was done while the authors were visiting the Mathematical Research and Conference Center Bedlewo in Autumn 2016 and the Erwin Schr\"odinger Institute for Mathematics and Physics (ESI) in Vienna in Autumn 2017. Both authors acknowledge support by the Polish Academy of Sciences;
Marcin Wnuk additionally acknowledges support by the German Academic Exchange Service (DAAD).

\end{document}